\documentclass[smallextended]{svjour3}[]

\usepackage[%
  font=small,
  labelfont=bf,
  tableposition=top
]{caption}
\usepackage{array}
\usepackage{mdwmath}
\usepackage{mdwtab}
\usepackage{subfloat}
\usepackage{eqparbox}
\usepackage{color}
\usepackage{graphicx}
\usepackage{slashbox}
\usepackage{subfig}
\usepackage{multirow}
\usepackage{amsmath, amssymb}
\usepackage{verbatim}
\usepackage{url}
\usepackage{booktabs}
\usepackage{algorithm,algorithmic}
\newcolumntype{x}[1]{%
>{\centering\hspace{0pt}}p{#1}}%
\usepackage{threeparttable}

\newcommand{\ignore}[1]{}
\everymath={\displaystyle}

\newcommand{\al}{{\alpha}}

\newcommand{\dt}{{\delta}}

\DeclareMathOperator{\diag}{diag}

\DeclareMathOperator{\card}{card}
\DeclareMathOperator{\spec}{spec}
\def\norm2#1{\|#1\|_2}

\def\mplus{\mathrel{%
  \ooalign{\raise.29ex\hbox{$\scriptscriptstyle\mathbf{+}$}\cr}}}

\begin{document}
%
\title{Matrix-free Interior Point Method for Compressed Sensing Problems}

\author{
Kimon Fountoulakis
\and
Jacek Gondzio\footnote{Supported by EPSRC Grant EP/I017127/1}
\and
Pavel Zhlobich
}

\institute{ 
Kimon Fountoulakis
  \at School of Mathematics and Maxwell Institute,
      The University of Edinburgh,
      Mayfield Road, Edinburgh EH9 3JZ,
      United Kingdom.
      \\\email{K.Fountoulakis@sms.ed.ac.uk}
      \\Tel.: +44 131 650 5083, Fax: +44 131 650 6553
\and 
Jacek Gondzio
  \at School of Mathematics and Maxwell Institute,
      The University of Edinburgh,
      Mayfield Road, Edinburgh EH9 3JZ,
      United Kingdom.
      \\\email{J.Gondzio@ed.ac.uk}
      \\Tel.: +44 131 650 8574, Fax: +44 131 650 6553
\and 
Pavel Zhlobich
  \at School of Mathematics and Maxwell Institute,
      The University of Edinburgh,
      Mayfield Road, Edinburgh EH9 3JZ,
      United Kingdom.
      \\\email{P.Zhlobich@ed.ac.uk}
      \\Tel.: +44 131 650 5044, Fax: +44 131 650 6553     
}

\maketitle

\begin{abstract}
We consider a class of optimization problems for sparse signal 
reconstruction which arise in the field of Compressed Sensing 
(CS). A plethora of approaches and solvers exist 
for such problems, for example GPSR, FPC\_AS, SPGL1,
NestA, $\mathbf{\ell_1\_\ell_s}$, PDCO to mention a few. 

CS applications lead to very well conditioned 
optimization problems and therefore can be solved easily by simple 
first-order methods. 
Interior point methods (IPMs) rely on the Newton method hence they use 
the second-order information. They have numerous advantageous features 
and one clear drawback: being the second-order approach they need 
to solve linear equations and this operation has (in the general 
dense case) an $\mathcal{O}(n^3)$ computational complexity. Attempts 
have been made to specialize IPMs to sparse reconstruction problems 
and they have led to interesting developments implemented in 
$\mathbf{\ell_1\_\ell_s}$ and PDCO softwares. We go a few steps 
further. First, we use the matrix-free IPM, 
an approach which redesigns IPM to avoid the need to explicitly 
formulate (and store) the Newton equation systems. Secondly, we exploit 
the special features of the signal processing matrices within 
the matrix-free IPM. Two such features are of particular interest: 
an excellent conditioning of these matrices and the ability to perform 
inexpensive (low complexity) matrix-vector multiplications with them. 

Computational experience with large scale one-dimensional 
signals confirms that the new approach is efficient 
and offers an attractive alternative to other state-of-the-art solvers.

\keywords{Matrix-free Interior Point, Preconditioned Conjugate Gradient, 
          Compressed Sensing, Compressive Sampling, $\ell_1$-regularization.}

\subclass{ 90C05, 90C06, 90C30, 90C25, 90C51 }

\end{abstract}

\section{Introduction}\label{sec:intro}

We are concerned with the solution of the incomplete system, $m < n$, of linear
equations
\begin{equation}\label{incomplsys}
A x = \hat b, 
\end{equation}
where $A\in\mathbb{R}^{m\times n}$, $x\in\mathbb{R}^n$, $\hat{b}\in\mathbb{R}^m$.
In particular, we are interested in the solution $x$ with the smallest 
possible number of nonzero elements, otherwise known as the sparsest 
solution $\hat{x}$. Such problems arise in the fields of Statistics
\cite{IEEEhowto:Miller} 
and Signal processing \cite{IEEEhowto:CandesCompSampl}. 

The sparsest solution $\hat{x}$ of system \eqref{incomplsys} can be found by solving the following problem:
\begin{equation}\label{zeronormform}
\begin{array}{ll}
\displaystyle\min_{x\in\mathbb{R}^{n}} & \|x\|_0 \\
\mbox{s.t.:} & Ax=\hat b, \\
\end{array}
\end{equation}
where $\|x\|_0=\{\#\mbox{ of nonzero entries in }x\}$ and "s.t." stands
for "subject to". The use of zero-norm makes 
the problem combinatorial and untractable in practice. Recent advances in the field
of CS show that in certain situations \cite{IEEEhowto:CandesCompSampl} exact recovery of the sparsest solution $\hat{x}$ 
of \eqref{incomplsys} can be achieved with an overwhelming probability by solving 
the following Basis Pursuit \cite{IEEEhowto:bpdn} problem:
\begin{equation}\label{onenormform}
\begin{array}{lll}
\multirow{2}{*}{BP: } \hspace{5.5mm} 
& \displaystyle\min_{x\in\mathbb{R}^{n}} & \|x\|_1 \\
& \mbox{s.t.:} & A x = \hat b, \\
\end{array}
\end{equation}
where $\|x\|_1=\sum\nolimits_{i=1}^n|x_i|$. 
The problem \eqref{onenormform} has a major advantage over \eqref{zeronormform}. 
Unlike the zero-norm objective in \eqref{zeronormform}, the $\ell_1$-norm objective 
in \eqref{onenormform} can be reformulated as a linear function and therefore 
the problem \eqref{onenormform} may be recast as a linear problem and becomes 
computationally tractable. Having a linear reformulation of (\ref{onenormform}), 
standard efficient optimization methods can be used to recover the sparsest 
solution $\hat{x}$. 

In real-life applications the right hand side of (\ref{incomplsys}) 
is often corrupted with noise and (\ref{incomplsys}) is replaced with:
\begin{equation}\label{incomplsysnoisy}
   A x = b = \hat b + e, 
\end{equation}
where $e\in\mathbb{R}^m$ denotes the error: we assume it has a normal 
distribution $e_i \sim \mathcal{N}(0,\sigma^2)$ $\forall\, i=1,2,\ldots,m$. For the noisy case (\ref{incomplsysnoisy}) 
the sparsest solution $\hat{x}$ can be found by solving one of the following 
problems:
\begin{subequations}\label{formulationsnoisy}
  \begin{equation}
    \begin{array}{lll}\label{formulationsnoisy1}
             \mbox{BPDN}:  \hspace{8.25mm} 
          & \displaystyle\min_{x\in\mathbb{R}^{n}} & \tau \|x\|_1 + \|Ax- b \|_2^2
    \end{array}
  \end{equation}
  \begin{equation}
    \begin{array}{lll}\label{formulationsnoisy2}
      \multirow{2}{*}{LASSO: } \hspace{20.2mm} 
          & \displaystyle\min_{x\in\mathbb{R}^{n}} & \|Ax- b \|_2 \\
      &\mbox{s.t.:}&  \|x\|_1 \le \epsilon_1 \\
    \end{array}
  \end{equation}
  \begin{equation}
    \begin{array}{lll}\label{formulationsnoisy3}
      \multirow{2}{*}{$\mbox{BP}_{\epsilon_2}$: } \hspace{17mm} 
          & \displaystyle\min_{x\in\mathbb{R}^{n}} & \|x\|_1 \\
      &\mbox{s.t.:}& \|Ax- b \|_2 \le \epsilon_2 \\
    \end{array}
  \end{equation}
\end{subequations}
where $\tau , \epsilon_1$ and $\epsilon_2$ are positive scalars that regulate
the sparsity and the upper bound on the noise error, respectively.
Problem (\ref{formulationsnoisy1}) is the well-known 
\textit{Basis Pursuit Denoising} introduced in \cite{IEEEhowto:bpdn}, 
problem (\ref{formulationsnoisy2}) 
is the \textit{Least Absolute Shrinkage and Selection Operator} (LASSO) used 
frequently in the field of computational statistics \cite{IEEEhowto:Miller}. 
It can be shown using Theorem 27.4 from \cite{IEEEhowto:Rockafellar} 
that the problems in (\ref{formulationsnoisy}) are equivalent for specific values 
of scalars $\tau$, $\epsilon_1$ and $\epsilon_2$.

Practical problems have large dimensions and off-the-shelf approaches such as the simplex 
method or the (standard) IPM are often impractical. 
However, matrices $A$ that appear in CS problems
display several attractive features which may be exploited within 
an optimization algorithm. This has created an interest in developing 
specialized approaches to solving such problems.

There have been various first-order methods developed for the solution of (\ref{onenormform}) and (\ref{formulationsnoisy}). Let us mention the ones known to be the most efficient. 
\begin{itemize}
\item \textit{Gradient Projection Sparse Reconstruction} GPSR \cite{IEEEhowto:gpsr} defines new variables $u,v\in\mathbb{R}^n$ such that
\begin{equation}\label{uvDef}
|x_i| = u_i + v_i \ \ \forall\, i=1,2,\ldots,n,
\end{equation}
where $u_i = \max(x_i,0)$ and $v_i = \max(-x_i,0)$. Then linearization of the $\ell_1$-norm is performed
\begin{equation}\label{lintechni}
\| x \|_1 = 1_n^\mathsf{T} u + 1_n^\mathsf{T} v,
\end{equation}
with $ u,v \ge 0$ and $1_n \in \mathbb{R}^n$ being a column vector of all ones.
Using the above linearization technique, GPSR solves the following constrained smooth reformulation of problem (\ref{formulationsnoisy1})
\begin{equation}\label{wrightBPDN}
\begin{array}{ll}
\displaystyle\min_{z\in\mathbb{R}^{2n}} & \tau1_{2n}^\mathsf{T}z + \frac{1}{2}\|F^\mathsf{T} z- b \|_2^2 \\
\mbox{s.t.:} & z \ge 0, \\
\end{array}
\end{equation}
where $z=[u\;;\;v]\in\mathbb{R}^{2n}$, $F^\mathsf{T} = [A \ -A]\in\mathbb{R}^{m\times 2n}$. Once optimal values of variables $u$ and $v$ are found the solution $x$ of the initial problem is retrieved by computing
\[
x = u - v. 
\]
The price for the linearization is that comparing to the initial BPDN problem \eqref{formulationsnoisy1} the dimension of the problem is doubled and $2n$ new non-negativity constraints are added. 
At each step of the algorithm a line search is performed along the negative gradient direction and the new iterate is projected to the feasible set defined by the imposed constraints $z\ge 0$.

\item \textit{Fixed Point Continuation Active Set} FPC\_AS \cite{IEEEhowto:fpc1} solves problem (\ref{formulationsnoisy1}). FPC\_AS is a two stage algorithm. At the first stage a shrinkage scheme is employed which aims to spot quickly the nonzero components of the sparse representation. Then the second stage is enabled to solve a smooth version of (\ref{formulationsnoisy1}) limited to the indexes of nonzero components found by the first stage of the algorithm.

\item \textit{Spectral Projected Gradient} SPGL1 \cite{IEEEhowto:spgl1} 
solves any of the problems (\ref{onenormform}), (\ref{formulationsnoisy2}) 
and (\ref{formulationsnoisy3}). The SPGL1 is a spectral projection gradient algorithm
which iteratively solves (\ref{formulationsnoisy2})  for some values of $\epsilon_1$, each approximate solution
of (\ref{formulationsnoisy2}) is used to build a root-finding problem, which is equivalent to (\ref{formulationsnoisy3}),
and is solved by employing a Newton method.

\item NestA \cite{IEEEhowto:Nesta} solves problem (\ref{formulationsnoisy3}) 
by using a variant of the Nesterov's smoothing gradient algorithm \cite{IEEEhowto:Nesterov}, 
which has been proved to have the optimal bound $\mathcal{O}({1}/{\epsilon})$ on the 
number of iterations, where $\epsilon$ is the required accuracy.
\end{itemize}

Independently there have been several attempts to design suitable IPM implementations. The most efficient among them, which can also handle large scale CS problems, are listed below.
\begin{itemize}
\item $\ell_1\_\ell_s$ algorithm \cite{IEEEhowto:l1ls} solves a constrained smooth reformulation of problem (\ref{formulationsnoisy1}) which allows a straightforward preconditioning of the Newton equation system that is solved with a conjugate gradient method.

\item PDCO algorithm \cite{IEEEhowto:pdco} solves regularized constrained smooth reformulations of problems (\ref{onenormform}) and (\ref{formulationsnoisy1}). The Newton equation system is solved by applying an LSQR ("Least Squares QR factorization") method.
\end{itemize} 
Both $\mathbf{\ell_1\_\ell_s}$ and PDCO have been demonstrated to be robust in comparison with other
IPM implementations. However, they are not as accurate and as fast as state-of-the-art first-order methods.

In this paper we present a primal-dual feasible IPM specialized to CS problems. Primal-dual because it iterates simultaneously on primal and dual variables of a smooth reformulation of problem (\ref{formulationsnoisy1}) and feasible because the smooth reformulation of (\ref{formulationsnoisy1}) consists only of conic constraints which are always satisfied. Primal-dual methods have been shown to have the best theoretical convergence properties \cite{kojimamegiddo} 
among various IPMs, but they also enjoy the best practical convergence \cite{IEEEhowto:JG-XXV,IEEEhowto:wrightbook}. Here we give a brief introduction of the structure of primal-dual IPM methods and we discuss important modifications that result in the proposed approach. 
The actual implementation used in this paper is given in Subsection \ref{sec:pdstipm}.

Primal-dual methods rely on Newton method to calculate primal-dual directions at each iteration. Newton method for primal-dual IPMs
finds roots for linearized KKT (Karush-Kuhn-Tucker) systems or their reduced versions known as augmented and normal equations systems. These systems arise as first-order optimality conditions of log-barrier primal-dual pairs. 
The linearized KKT systems, referred as Newton linear systems, can be solved in two ways,
\begin{itemize}
\item by employing a direct linear solver, or
\item by using an iterative solver, such as Krylov subspace methods \cite{mybib:bookKelley}.
\end{itemize}  

The first option delivers a very robust primal-dual IPM where exact Newton directions are calculated. 
Despite its robustness this approach has the potential drawback of being computationally expensive.
Especially in the case when the Newton linear system does not have an exploitable sparsity pattern and the computational effort per iteration reaches $\mathcal{O}(n^3)$. 

The second option involves the use of approximate Newton directions. Although this might slightly increase the number of IPM iterations \cite{jacekinexact,lumonteiro}, 
one hopes that the decreased computational effort per iteration should offset such a disadvantage. The performance of iterative methods depends on the spectral
properties of the Newton linear system \cite{mybib:bookKelley} and benefit from the use of
appropriate preconditioning techniques which cluster the eigenvalues
of the Newton linear system. If the Newton linear system is ill-conditioned and no low-cost preconditioner is applicable, then a direct approach might be more efficient.
To conclude, a criterion to select between the two approaches of solving the Newton linear systems should take into account 
\begin{enumerate}
\item the sparsity pattern of the systems,
\item the existence of a computationally inexpensive preconditioner, 
\item the memory requirements of storing problem's data,
\item the existence of fast matrix-vector product implementations with the matrix of the linear system to be solved.
\end{enumerate}

 In this paper, we focus on the situation where there is no particular sparsity pattern, the memory requirements can be high but conditions $2$ and $4$ are satisfied. 
 For this reason, a preconditioned conjugate gradient method is more attractive than a direct method. Indeed, in the approach proposed in this paper, at each step of the primal-dual IPM the preconditioned conjugate gradient method is applied to compute an approximate Newton direction. Since we rely on an iterative method for linear algebra, the proposed primal-dual IPM is matrix-free \cite{IEEEhowto:Jacekmf}, i.e. the explicit problem formulation is avoided and the measurement matrix $A$ is used only as an operator to produce results of matrix-vector products $Ax$ and $A^\mathsf{T}y$.
Although matrices $A$ used in CS can be completely dense, i.e. Gaussian, partial Fourier, partial DCT (Discrete Cosine Transform), partial DST (Discrete Sine Transform), Haar wavelets etc, they do have interesting (exploitable) features. Arguments $3$ and $4$ are satisfied because for many measurement matrices that appear in sparse signal reconstruction problems there are \emph{super-fast algorithms (e.g. $\mathcal{O}(n)$ or $\mathcal{O}(n\log n)$ complexity) for multiplication by a vector}. 
 For example, for Fourier, DCT and DST matrices there exists the FFTW implementation ("Fastest Fourier Transform in the West") \cite{fftw} with complexity $\mathcal{O}(n\log n)$, for Haar wavelet and Noiselet matrices there exist algorithms of complexity $\mathcal{O}(n)$, see \cite{nasirkami} and \cite{noiselets}, respectively.
 Finally, to satisfy argument $2$ we propose a preconditioner efficient on certain problems that is based on the fact that \emph{sub-matrices of $A$ with a given number of columns are uniformly well-conditioned} (this is called the Restricted Isometry Property, see the discussion in Section~\ref{sec:CSMatrices}). 
 
The objective of our 
developments is to design an IPM which preserves the main advantage of IPM, that is, it converges in merely a few iterations, and removes the main drawback of IPM, that is, avoids expensive computations of the Newton direction. Ideally, we would like to solve the CS problems in $\mathcal{O}(\log n)$ IPM iterations and keep the cost of a single IPM iteration as low as possible and 
not exceeding $\mathcal{O}(n \log n)$.

The paper is organized as follows. In Section~\ref{sec:CSMatrices} we discuss the particular features of CS matrices that are exploited in our approach. In Section~\ref{sec:SepForm} we reformulate sparse recovery optimization problems \eqref{onenormform} and \eqref{formulationsnoisy1} to make them suitable for matrix-free IPM. Section~\ref{sec:PCG} concerns finding approximate Newton directions required at each step of the IPM. We calculate the normal equations system formulation of the above stated problem and analyze its properties. We propose an efficient preconditioner that can be used in the preconditioned conjugate gradient method. We prove that under certain conditions (that are satisfied in practice) eigenvalues of the preconditioned matrix are well clustered around 1. In Section~\ref{compexpsec} we compare the proposed matrix-free IPM with other state-of-the-art first and second-order solvers.

\section{Properties of Compressed Sensing Matrices}\label{sec:CSMatrices}
Matrices which appear in sparse reconstruction problems originate from different bases in which signals are represented. What they all have in common are the conditions that guarantee recoverability of the sparsest solution of \eqref{incomplsys} by means of the $\ell_1$-norm minimization \eqref{onenormform}. The restricted isometry property (RIP) \cite{IEEEhowto:CandesRombergTaoStable} is one of such conditions which shows how efficiently a measurement matrix captures information about sparse signals.

\begin{definition}
The restricted isometry constant $\dt_k$ of a matrix $A\in\mathbb{R}^{m\times n}$ is defined as the smallest $\dt_k$ such that
\begin{equation}\label{eq:rip1}
(1-\dt_k)\norm2{x}^2\leq\norm2{Ax}^2\leq(1+\dt_k)\norm2{x}^2
\end{equation}
for all at most $k$-sparse $x\in\mathbb{R}^n$.
\end{definition}

In words, for small $\delta_k$, statement \eqref{eq:rip1} requires that all column sub-matrices of $A$ with at most $k$ columns are well-conditioned. Informally, $A$ is said to satisfy the RIP if $\dt_k$ is small for a reasonably large $k$. The next theorem due to \cite{foucart2010note} establishes the relation between the RIP property and the sparse recovery.

\begin{theorem}\label{thm:RIP_bound}
Every $k$-sparse vector $x\in\mathbb{R}^{n}$ satisfying $A x = \hat b$ is the unique solution of \eqref{onenormform} if
\[
\dt_{2k}<\frac{3}{4+\sqrt{6}}\approx 0.4652.
\]
\end{theorem}

The restricted isometry property also implies stable recovery by $\ell_1$-norm minimization for vectors that can be well approximated by sparse ones, and it further implies robustness under noise on the measurements \cite{IEEEhowto:CandesRombergTaoStable}. 

RIP is a very restrictive condition that depends on the size of the measurement matrix $A$. Clearly, the more columns $n$ matrix $A$ has (the larger the size of the vector $x$ to recover) the larger $\dt_k$ in \eqref{eq:rip1} is (the harder it is to guarantee sparse recovery). On the other hand, number of rows $m$ of $A$ is the number of measurements taken and, hence, the RIP constant $\delta_k$ decreases with $m$. Currently known measurement matrices satisfying RIP with small number of measurements fall into two categories \cite{rudelson2008sparse}: (i) random matrices with i.i.d. sub-Gaussian variables, e.g., normalized i.i.d. Gaussian or Bernoulli matrices; (ii) random partial bounded orthogonal matrices obtained by choosing $m$ rows uniformly at random from a normalized $n\times n$ Fourier or Walsh-Hadamard transform matrices. Number of measurements required to satisfy the RIP property for both classes of matrices is given in the table below.

\begin{table}[ht]
\caption{List of measurement matrices that have been proven to satisfy RIP}\label{tbl:RIP}
\renewcommand{\arraystretch}{1.2}
\centering
\begin{tabular}{|l|l|l|}
\hline
$m\times n$ measurement matrix & RIP regime & references \\
\hline \hline
Gaussian & $m \geq Ck\log n$ & \cite{baraniuk2008simple,rudelson2008sparse} \\ \hline \hline
partial Fourier & $m \geq Ck\log^4 n$ & \cite{rudelson2008sparse} \\
\hline
\end{tabular}
\end{table}

Although it follows from Table \ref{tbl:RIP} that Gaussian matrices are optimal for sparse recovery, they have limited use in practice because many applications impose structure on the matrix. Furthermore, recovery algorithms are significantly more efficient when the matrix admits a fast matrix-vector product. Due to the two former practical reasons, and since we are dealing with large-scale CS applications we limit ourselves to applications with measurement matrices $A$ that 
\begin{itemize}
\item are not stored explicitly,
\item admit a low-cost matrix-vector product with $A$ (e.g. $\mathcal{O}(n\log n)$ or $\mathcal{O}(n)$).
\end{itemize}

An important broad class of CS matrices comes from random sampling in bounded orthonormal systems. Partial Fourier matrix mentioned earlier is just one example of this type. Other examples are matrices related to systems of real trigonometric polynomials (partial discrete cosine (DCT) and discrete sine (DST) matrices), Haar wavelets and noiselets. Quite often in applications a signal is sparse with respect to a basis different from the one in which measurements are made. Then it is said that a measurement/sparsity pair is given \cite{IEEEhowto:CandesCompSampl}. Assume that a vector $z$ is sparse with respect to the basis of columns of a unitary matrix $\Psi$ (\emph{sparsity matrix}), i.e. $z=\Psi x$ for a $k$-sparse vector $x$. Further, assume that $z$ is sampled with respect to the basis of columns of a unitary matrix $\Phi$ (\emph{measurement matrix}): $y = R_{m}\Phi^\mathsf{T}z$, where $R_m$ is a random sampling operator which satisfies $R_mR_m^\mathsf{T}=I$. Hence, matrix $A$ in \eqref{incomplsys} is equal to $R_{m}\Phi^\mathsf{T}\Psi$ and its rows are orthonormal:
\begin{equation}\label{eq:orthonormality}
AA^\mathsf{T} = I_m.
\end{equation}
The recoverability property of matrix $A$ depends on the value of the so-called \emph{mutual coherence} $\mu(\Phi,\Psi)$ of the measurement/sparsity pair (see \cite{IEEEhowto:DonohoHuo}):
\begin{equation}\label{eq:mut_coh}
\mu(\Phi,\Psi)=\sqrt{n}\displaystyle\max_{i,j}|\left\langle \phi_i, \psi_j \right\rangle|,
\end{equation}
where $\phi_i,\psi_i$ are the $i^{th}$ columns of matrices $\Phi,\Psi$, respectively. Coherence simply measures the largest correlation between any two elements of $\Phi$ and $\Psi$. Next theorem due to \cite{candes2007sparsity} shows that the smaller the value of mutual coherence the better the recoverability property of matrix $A$.

\begin{theorem}
Fix $z\in\mathbb{R}^n$ and suppose that the coefficient sequence $x$ of $z$ in the unitary $n\times n$ basis $\Psi$ is $k$-sparse. Select $m$ measurements in the unitary $n\times n$ $\Phi$ domain uniformly at random. Then if
\begin{equation}\label{eq:mut_coh_cond}
m\geq Ck\mu(\Phi,\Psi)^2\log (n/p) \quad \mbox{and} \quad m\ge C' \log^2(n/p)
\end{equation}
for some positive constants $C,C'$, then with overwhelming probability exceeding $1-p$, the vector $x$ is the unique solution to the $\ell_1$-minimization problem \eqref{onenormform} with $A=R_{m}\Phi^\mathsf{T}\Psi$, 
where $R_mR_m^\mathsf{T}=I$ and $A$ has orthonormal rows (\ref{eq:orthonormality}).
\end{theorem}

Let us note that condition \eqref{eq:mut_coh_cond} differs from those given in Table \ref{tbl:RIP}. Conditions in Table \ref{tbl:RIP} ensure that once the random matrix is chosen, then with high probability all sparse signals can be recovered (\emph{uniform recovery}). Although, \eqref{eq:mut_coh_cond} only guarantees that each fixed sparse signal can be recovered with high probability using a random draw of the matrix (\emph{nonuniform recovery}).

To conclude, CS matrices have many useful properties that must be taken into account in the development of an efficient matrix-free IPM solver. In the current paper we make use only of the most general of them that are satisfied by every CS matrix. First, we weaken a little bit the condition of orthonormality \eqref{eq:orthonormality} to include random matrices such as Gaussian and Bernoulli:

\begin{itemize}
\item\label{P1}\textbf{P1:}
Rows of matrix $A$ are close to orthonormal, i.e. there exists a small $\dt$ such that
\begin{equation}\label{eq:loose_orthonormality}
\norm2{AA^\mathsf{T} - I_m}\leq\dt.
\end{equation}
\end{itemize}

Restricted isometry property \eqref{eq:rip1} on the contrary assumes that columns of $A$ are normalized. So, our interpretation of the RIP property that will be used throughout the paper is as follows.

\begin{itemize}
\item\label{P2}\textbf{P2:}
Every $k$ columns of $A$ with $k \ll m$ are almost orthogonal and have similar norms, i.e. for every matrix $B$ composed of arbitrary $k$ columns of $A$
\begin{equation}\label{eq:rip2}
\left\|\frac{n}{m} B^\mathsf{T}B - I_k\right\|_2\leq\dt_k.
\end{equation}
\end{itemize}
By treating property \textbf{P2} as the chosen RIP, the bound for the RIP constant in Theorem \ref{thm:RIP_bound} which relies on RIP in (\ref{eq:rip1}) will change. The following theorem is a modified version of Theorem \ref{thm:RIP_bound} when property \textbf{P2} is used as a RIP.
\begin{theorem}\label{thm:RIP_bound2}
Every $k$-sparse vector $x\in\mathbb{R}^{n}$ satisfying $A x = \hat b$ is the unique solution of \eqref{onenormform} if
\[
\dt_{2k}<\frac{3\frac{m}{n}}{1+3\frac{m}{n}+\sqrt{6}},
\]
where $\delta_{2k}$ is the minimum constant such that property \textbf{P2} holds for every $2k$ columns of matrix $A$, denoted by matrix $B$ in \textbf{P2}.
\end{theorem}
\begin{proof}
Let $x\in\mathbb{R}^{n}$ have $k$ nonzero components and $B$ in \textbf{P2} be any $k$ column submatrix of $A$. Then from \textbf{P2} it follows that
 \begin{align}
\frac{m}{n}(1-\dt_k) \|x\|_2^2 &\le \left\| A x\right\|_2^2  \leq \frac{m}{n}(1+\dt_k) \|x\|_2^2. &  \label{bfkim2}
\end{align}
Proposition $2$ in \cite{foucart2010note} gives bounds for $\delta_{2k}$ by using the RIP in (\ref{eq:rip1}). In our case, 
we replaced the RIP in (\ref{eq:rip1}) with (\ref{bfkim2}). Therefore, the four modified conditions for $\delta_{2k}$ in Proposition $2$ in \cite{foucart2010note} which guarantee that every $k$-sparse
vector $x\in\mathbb{R}^n$ which satisfies $A x = \hat b$ is the unique solution of \eqref{onenormform}, take the following forms:
\begin{align}
1) \ & \delta_{2k} < \frac{1}{2}\frac{m}{n}& \mbox{when} \ \ &k=1&  \nonumber \\
2) \  & \delta_{2k} < \frac{3\frac{m}{n}}{(1+3\frac{m}{n} + \sqrt{(6k-2r)/(k-1)})}& \mbox{when} \ \  &k=3\omega + r \ \mbox{and} \ 1\le r\le 3 \nonumber \\
3) \  & \delta_{2k} < \frac{4\frac{m}{n}}{(1+4\frac{m}{n} + \sqrt{(12k-3r)/(k-1)})}& \mbox{when} \ \  &k=4\omega + r \ \mbox{and} \ 1\le r\le 4 \nonumber \\
4) \  & \delta_{2k} < \frac{2\frac{m}{n}}{(1+2\frac{m}{n} + \sqrt{1 + k/(8\omega + \lfloor8r/5\rfloor})}& \mbox{when} \ \  &k=5\omega + r \ \mbox{and} \ 1\le r\le 5, \nonumber 
\end{align}
where $\omega=0,1,\ldots$ is an integer variable.
Table \ref{weakRIP} shows with bold font which condition of the above four is the weakest for $2\le s\le 8$. This table is equivalent of the table in proof of Theorem 1
in \cite{foucart2010note}. However, in \cite{foucart2010note} the table has exact values, where our Table \ref{weakRIP} has functions depending on the ratio ${m}/{n}$ instead.
\begin{table}
\caption{Weakest RIP constant values for sparsity level $2 \le k \le 8$}
\renewcommand{\arraystretch}{2.3}
\centering
\begin{tabular}{|c|c|c|c|}
\hline
            & Case $2$           										  & Case $3$         									          & Case $4$  \\ \hline \hline
$k=2$ & $\frac{3\frac{m}{n}}{1+3\frac{m}{n}+2\sqrt{2}}$                                   & $\frac{4\frac{m}{n}}{1+4\frac{m}{n}+3\sqrt{2}}$                  		 & $\mathbf{\frac{2\frac{m}{n}}{1+2\frac{m}{n}+\sqrt{\frac{5}{3}}}}$         \\ \hline\hline
$k=3$ & $\mathbf{\frac{3\frac{m}{n}}{1+3\frac{m}{n}+\sqrt{6}}}$                     & $\frac{4\frac{m}{n}}{1+4\frac{m}{n}+3\sqrt{\frac{3}{2}}}$   		 & $\frac{2\frac{m}{n}}{1+2\frac{m}{n}+\frac{1}{2}\sqrt{7}}$         \\ \hline\hline
$k=4$ & $\frac{3\frac{m}{n}}{1+3\frac{m}{n}+\sqrt{\frac{22}{3}}}$   		  & $\mathbf{\frac{4\frac{m}{n}}{1+4\frac{m}{n}+2\sqrt{3}}}$                   & $\frac{2\frac{m}{n}}{1+2\frac{m}{n}+\sqrt{\frac{5}{3}}}$         \\ \hline\hline
$k=5$ & $\frac{3\frac{m}{n}}{1+3\frac{m}{n}+\sqrt{\frac{13}{2}}}$   		  & $\frac{4\frac{m}{n}}{1+4\frac{m}{n}+\frac{1}{2}\sqrt{57}}$ 		 & $\mathbf{\frac{2\frac{m}{n}}{1+2\frac{m}{n}+\frac{1}{2}\sqrt{\frac{13}{2}}}}$         \\ \hline\hline
$k=6$ & $\frac{3\frac{m}{n}}{1+3\frac{m}{n}+\sqrt{6}}$                     		  &$\frac{4\frac{m}{n}}{1+4\frac{m}{n}+\sqrt{\frac{66}{5}}}$    		 & $\mathbf{\frac{2\frac{m}{n}}{1+2\frac{m}{n}+\sqrt{\frac{5}{3}}}}$         \\ \hline\hline
$k=7$ & $\frac{3\frac{m}{n}}{1+3\frac{m}{n}+2\sqrt{\frac{5}{3}}}$   		  &$\mathbf{\frac{4\frac{m}{n}}{1+4\frac{m}{n}+\frac{5}{\sqrt{2}}}}$       & $\frac{2\frac{m}{n}}{1+2\frac{m}{n}+3\sqrt{\frac{2}{11}}}$         \\ \hline\hline
$k=8$ & $\frac{3\frac{m}{n}}{1+3\frac{m}{n}+2\sqrt{\frac{11}{7}}}$ 		  & $\mathbf{\frac{4\frac{m}{n}}{1+4\frac{m}{n}+2\sqrt{3}}}$                   &$\frac{2\frac{m}{n}}{1+2\frac{m}{n}+\sqrt{\frac{5}{3}}}$         \\ \hline
\end{tabular}
\label{weakRIP}
\end{table}
Using the same arguments as in proof of Theorem $1$ in \cite{foucart2010note} and Table \ref{weakRIP} we conclude that 
every $k$-sparse vector $x\in\mathbb{R}^{n}$ satisfying $A x = \hat b$ is the unique solution of \eqref{onenormform} if
\[
\dt_{2k}<\frac{3\frac{m}{n}}{1+3\frac{m}{n}+\sqrt{6}}.
\]
This completes the proof.
\end{proof}
Comparing the two bounds of the RIP constants in Theorems \ref{thm:RIP_bound} and \ref{thm:RIP_bound2} we observe that the former is smaller, see Figure \ref{ripfigcom}.
For the purpose of the proposed preconditioner, discussed in Section~\ref{sec:PCG}, the smaller bound on $\delta_{2k}$ in Theorem \ref{thm:RIP_bound2}
results in tighter bounds of the spectral properties of the preconditioned systems. The former is an advantage of property \textbf{P2} against RIP in (\ref{eq:rip1}), proved in Lemma \ref{lem:ne2_prec}.
However, property \textbf{P2} and Theorem \ref{thm:RIP_bound2} result in a limitation of the maximum number of sparsity $k$ for which problem \eqref{onenormform}
guarantees an exact recovery of the sparsest solution of $Ax= \hat b$. Fortunately, both results in Theorems \ref{thm:RIP_bound} and \ref{thm:RIP_bound2} are rather
pessimistic. It has been shown in \cite{BlanchardCartis} that RIP conditions of the form (\ref{eq:rip1}) and their scaled versions (\textbf{P2}) or (\ref{bfkim2}) provide worst case scenarios
of $\delta_{2k}$ and consequently of the sparsity level $k$ such that problem \eqref{onenormform} guarantees exact sparse recovery. To support the former argument,
we refer the reader to \cite{DonohoJarred:PUT}, where it is shown that for Gaussian measurement matrices the average maximum sparsity level $k$ that is guaranteed to be reconstructable by \eqref{onenormform}
is much greater than the one shown in Theorems \ref{thm:RIP_bound} and \ref{thm:RIP_bound2}. 
Moreover, it has been shown empirically in \cite{DonohoJarred:PUT} that approximately the same result holds for various types of measurement matrices $A$, i.e. partial Fourier, partial Hadamard, Bernoulli etc. 
In Subsection~\ref{optransition} it is shown that the proposed algorithm satisfies approximately the average maximum sparsity $k$ shown in \cite{DonohoJarred:PUT}.
Therefore, we conclude that by replacing RIP (\ref{eq:rip1}) with property \textbf{P2}:
\begin{itemize}
\item improved bounds on the spectral properties of the preconditioned systems in Section~\ref{sec:PCG} are obtained,
\item a better approximation of matrix $B^\mathsf{T}B$ with a scaled diagonal $\rho I$ by choosing appropriate constant $\rho$ is possible and
\item the empirical average reconstruction properties as shown in Subsection~\ref{optransition} are maintained.
\end{itemize}
\begin{figure}%
\centering
\subfloat{%
\includegraphics[scale=0.45]{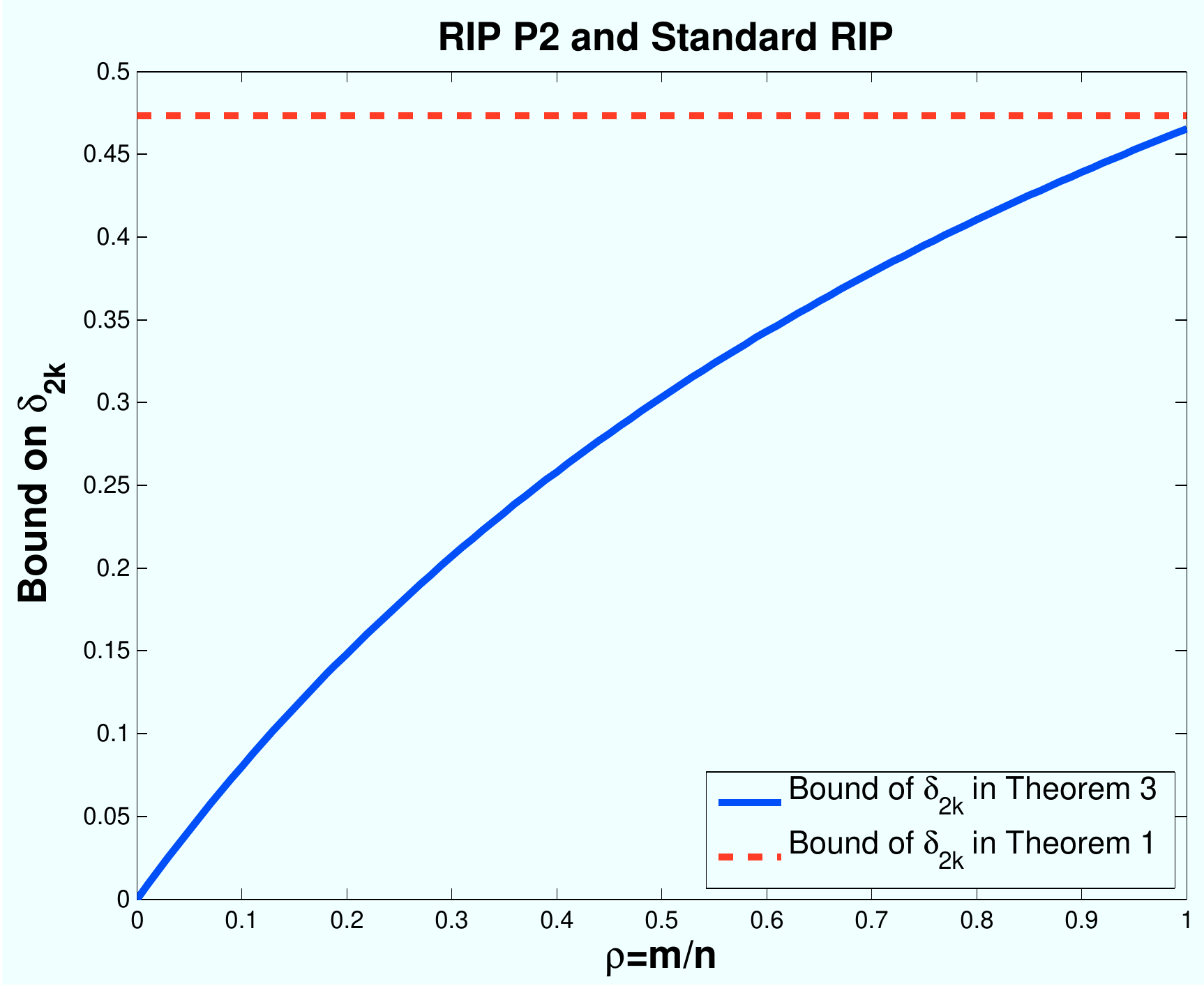}}
\caption{Comparison on the bounds of $\delta_{2k}$ constants from Theorems \ref{thm:RIP_bound} and \ref{thm:RIP_bound2}}
\label{ripfigcom}%
\end{figure}
\section{Primal--Dual Problems in Matrix-free IPM}\label{sec:SepForm}
Non-smooth Basis Pursuit \eqref{onenormform} and Basis Pursuit Denoising \eqref{formulationsnoisy1} optimization problems can be reformulated into equivalent linear and convex quadratic problems, respectively. This is achieved via linearization of the non-smooth $\ell_1$-norm in the objective function.

After reformulating the BPDN problem (\ref{formulationsnoisy1}) to (\ref{wrightBPDN}) as proposed in \cite{IEEEhowto:gpsr} for GPSR algorithm, we solve the latter using a primal--dual IPM. The reader interested in the theory of primal-dual IPMs is referred to the book of Wright \cite{IEEEhowto:wrightbook}. Aspects of practical implementation have been addressed in a recent survey \cite{IEEEhowto:JG-XXV}. A description of the primal-dual IPM used in this paper is given in Section~\ref{sec:pdstipm}. For the primal problem \eqref{wrightBPDN} of interest the dual is
\begin{equation}\label{dualquadraticBPDN}
\begin{array}{llll}
\multirow{4}{*}{Dual Sep.:} \hspace{9mm} & \displaystyle\max_{z,s\in\mathbb{R}^{2n}} & -\frac{1}{2} z^\mathsf{T}FF^\mathsf{T} z \\
& \mbox{s.t.:} & c + FF^\mathsf{T}z -s = 0 & \\
& & z,s \ge 0 &
\end{array}
\end{equation}
where
$
\begin{array}{c}
 c = \left[
 \begin{array}{c}
     \tau1_n - A^\mathsf{T} b \\
     \tau1_n + A^\mathsf{T} b
 \end{array} 
 \right]\in\mathbb{R}^{2n}.
\end{array}
$

At each step of the primal--dual IPM applied to the primal-dual pair \eqref{wrightBPDN} and \eqref{dualquadraticBPDN} the corresponding Newton direction $(\Delta z,\Delta s)$ is computed by solving the following system of linear equations:
\begin{equation}\label{sepregaugeq}
\setlength{\extrarowheight}{0.3ex}
\left[\begin{array}{cc}
FF^\mathsf{T}& -I_{2n}\\
S& Z \\
\end{array}\right]
\times
\left[\begin{array}{c}
\Delta z \\ \Delta s \\
\end{array}\right]
=
\left[\begin{array}{c}
f_z \\ f_s \\
\end{array}\right],
\end{equation}
where $S$ and $Z$ are diagonal matrices with vectors $s$ and $z$ on the diagonal, respectively, $I_{2n}$ denotes an identity matrix of dimension $2n$ and 
\begin{equation}\label{fzfs}
 f_z   = s - c - FF^\mathsf{T}z, \quad
 f_s   = \sigma\mu e - ZS1_{2n},
\end{equation}
$\mu={z^\mathsf{T}s}/{(2n)}$ is the barrier term of the IPM and $0\le\sigma\le 1$ the centering parameter.
In the matrix-free framework the dual variables $\Delta s$ in \eqref{sepregaugeq} are eliminated to get:
\begin{subequations}\label{redNewton}
\begin{equation}
\label{redNewtonA}
(\Theta^{-1} + FF^\mathsf{T})\Delta z = f_z + Z^{-1}f_s,
\end{equation}
\begin{equation}
\label{redNewtonB}
\hspace{2.9cm} \Delta s =  Z^{-1}f_s - \Theta^{-1}\Delta z.
\end{equation}
\end{subequations}
where $\Theta = S^{-1}Z \in \mathbb{R}^{2n \times 2n}$. The reduced Newton system (\ref{redNewtonA}), also known as augmented system, is solved by an appropriate preconditioned iterative method for which \textit{only matrix-vector product with the constraint matrix $F$ is allowed}. Thus, the matrix-free IPM approach has two major components:
\begin{itemize}
\item iterative solver for the augmented system,
\item special-purpose preconditioner that exploits matrix structure.
\end{itemize}
The next section addresses these two issues.

\section{Preconditioned Conjugate Gradient Method}\label{sec:PCG}
The system \eqref{redNewtonA} has a symmetric positive definite matrix and the conjugate gradient (CG) method can be employed to solve it in a matrix-free regime. 
However, the convergence of the CG method can be too slow when a matrix is ill-conditioned and/or its eigenvalues are not clustered. 
In this section we discuss an efficient spectrally-equivalent diagonal matrix preconditioner for (\ref{redNewtonA}). 
In particular, we give theoretical and practical justification of our approach to fast iterative solution of the system.

The proposed preconditioner for the system of equations \eqref{redNewtonA} is based on the exploitation of general properties of CS matrices and the behavior of the $\Theta$ matrix in \eqref{redNewtonA} close to optimality. 
Let us recall that in the notation of primal--dual pair \eqref{wrightBPDN}--\eqref{dualquadraticBPDN}, variable $s \in \mathbb{R}^{2n}$ is a Lagrange multiplier associated with the non-negativity constraint $z \geq 0$. Hence, at optimality $s_j z_j = 0$ $\forall\, j = 1,2,\dots, 2n$. IPMs force the convergence to the optimal solution by perturbing this condition $s_j z_j = \mu$ $\forall\, j$, where $\mu$ is the barrier term of the IPM, 
and gradually reducing the perturbation $\mu$ to zero. At optimality indices $j \in \{ 1, 2,\dots, 2n \}$ are split into two disjoint sets: 
\begin{equation}
\begin{gathered}
\mathcal{B} = \{j \ | \ z_j \to z_j^{*} > 0, s_j \to s_j^{*} = 0 \}\\ 
\mbox {and}\\
\mathcal{N} = \{j \ | \ z_j \to z_j^{*} = 0, s_j \to s_j^{*} > 0 \}
\end{gathered} 
\end{equation}
that determine the activity of constraints. This partitioning has highly undesirable consequences for the diagonal scaling matrix $\Theta = S^{-1} Z$. Indeed, when $\mu$ approaches zero, for indices $j \in \mathcal{B}$, $\Theta_j$ goes to infinity and for indices $j \in \mathcal{N}$, $\Theta_j$ 
goes to zero. 

Recall that $z=[u\;;\;v]$, where $u$ and $v$ are the positive and negative components of vector $x$ (see \eqref{uvDef}), respectively. For sparse signals there are merely $k$ ($k \ll 2n$) nonzero components in the optimal solution. The positive ones will contribute a nonzero element in $u$ and the negative ones will contribute a nonzero element in $v$. At optimality the cardinality of set $\mathcal{B}$ is $k$. Hence, at later iterations of an IPM
\begin{equation}\label{eq:thetas_split}
\begin{aligned}
&\Theta_i \gg 1\quad \forall\,i\in\mathcal{B},\quad\; \card{\mathcal{B}}=k,\\
&\Theta_i \ll 1\quad \forall\,i\in\mathcal{N},\quad \card{\mathcal{N}}=2n-k.
\end{aligned} 
\end{equation}

Let us now return to the question of preconditioning of the system of equations \eqref{redNewtonA}. Its matrix is
\begin{equation}\label{eq:ne2_matrix}
H = \Theta^{-1} + FF^\mathsf{T}.
\end{equation}
The behavior of matrix $\Theta$ near optimality is described by \eqref{eq:thetas_split}. It is clear that matrix $\Theta^{-1}$ has many large entries and only few small entries well before the IPM reaches the optimal solution. Let us introduce a number $C\gg1$ that separates entries of $\Theta^{-1}$ of different magnitudes:
\begin{equation}\label{eq:separator_C}
\#(\Theta_{j}^{-1}<C)=l.
\end{equation}
Here $l$ is just the number of small entries in $\Theta^{-1}$ and may be different from the sparsity $k$ of the optimal solution. In the regime $l < m$, the second term $FF^\mathsf{T}$, whose rank is exactly $m$, works as a low-rank pertubation for the matrix $\Theta^{-1}$ in \eqref{eq:ne2_matrix}. Since, in Frobenius norm the first term $\Theta^{-1}$ dominates the second term $FF^\mathsf{T}$, we propose to replace $FF^\mathsf{T}$ in the preconditioner by a simple approximant. First, let us write system's matrix of \eqref{redNewtonA} in the block form by using the facts that $\Theta=\diag(\Theta_u, \Theta_v)$ and $F^\mathsf{T} = [A \ -A]$:
\begin{equation}\label{eq:ne2_H}
H =
\begin{bmatrix}
\Theta_u^{-1} & \\
& \Theta_v^{-1}
\end{bmatrix}
+
\begin{bmatrix}
A^\mathsf{T}A & -A^\mathsf{T}A \\
-A^\mathsf{T}A & A^\mathsf{T}A \\
\end{bmatrix}.
\end{equation}
Our preconditioner is based on the approximation of $A^\mathsf{T}A$ by the closest (in Frobenius norm) scaled identity matrix $\rho I_n$, $\rho=m/n$:
\begin{equation}\label{eq:ne2_P}
P=
\begin{bmatrix}
\Theta_u^{-1}+\rho I_n & -\rho I_n \\
-\rho I_n & \Theta_v^{-1}+\rho I_n \\
\end{bmatrix}.
\end{equation}

To simplify the analysis of the preconditioner, we first consider the case of $n\times n$ matrices $H$ and $P$ rather than block $2n\times 2n$ ones as defined by \eqref{eq:ne2_H} and \eqref{eq:ne2_P}. The following lemma establishes spectral properties of the preconditioned matrix $P^{-1}H$ in the non-block case.

\begin{lemma}\label{lem:ne2_prec}
Define matrix $H$ as
\[
H = \Theta^{-1} + A^\mathsf{T}A,
\]
where $\Theta=\diag(\Theta_1,\Theta_2,\ldots,\Theta_n)$ --- diagonal $n\times n$ matrix with $\Theta_j>0$, and $A$ --- $m\times n$ matrix with $m \leq n/2$. Let $C$ be any positive constant and $l$ be defined as in (\ref{eq:separator_C}), $\#(\Theta_{j}^{-1}<C)=l$. Additionally, let $A$ satisfy property \textbf{P2} defined on page \pageref{P2} for $k=l$ with some constant $\delta_l$. If matrix $A$ has orthonormal rows (\ref{eq:orthonormality}),
then the eigenvalues of matrix $H$ preconditioned by matrix $P$:
\[
P = \Theta^{-1} + \rho I_n,\quad \rho=m/n
\]
are clustered around 1, i.e.
\begin{equation}\label{eq:clustering}
|\lambda-1|\leq\dt_l+\frac{1}{4}\frac{(3-\rho)^2}{\rho\dt_l C}\quad\forall\, \lambda\in\spec(P^{-1}H),
\end{equation}
If matrix $A$ has nearly orthonormal rows, i.e. satisfies \textbf{P1} defined on page \pageref{P1}, then
\[
|\lambda-1|\leq  \dt_l+\frac{1}{4}\frac{(1+\delta-\rho+2\sqrt{1+\delta})^2}{\rho\dt_l C}\quad\forall\, \lambda\in\spec(P^{-1}H),
\] 
where $\delta$ has been defined in \textbf{P1}.
\end{lemma}
\begin{proof}
Let $C$ be any positive constant, then the following two disjoint sets of indices can be defined:
\[
{\mathcal{B}_C} = \{j\in\{1,2,\ldots,n\}:\; \Theta_{j}^{-1} < C\},\quad {\mathcal{N}_C} = \{1,2,\ldots,n\}\setminus{\mathcal{B}_C}
\]
Let $B$ and $N$ be matrices of columns of $A$ with indices from ${\mathcal{B}_C}$ and ${\mathcal{N}_C}$, respectively. Without loss of generality we can assume that ${\mathcal{B}_C}$ are the first $l$ indices, then
\[
A = [B \; N],\quad B\in\mathbb{R}^{m\times l},\quad N\in\mathbb{R}^{m\times (n-l)}.
\]

Let $\lambda$ be an eigenvalue of the preconditioned matrix $P^{-1}H$ corresponding to an eigenvector $v=[v_{\mathcal{B}_C}\;;\;v_{\mathcal{N}_C}]$ of norm one, then
\begin{equation}
P^{-1}Hv=\lambda v\quad\Longleftrightarrow\quad(H-P)v=\tau Pv,\quad \tau=\lambda-1,
\end{equation}
or, in the block form,
\begin{equation}\label{eq:precH1}
\setlength{\extrarowheight}{0.3ex}
\left[\begin{array}{c|c}
B^\mathsf{T}B-\rho I_l & B^\mathsf{T}N \\
\hline\vspace{0.01cm}
N^\mathsf{T}B & N^\mathsf{T}N-\rho I_{n-l}
\end{array}\right]
\left[\begin{array}{c}
v_{\mathcal{B}_C} \\
\hline
v_{\mathcal{N}_C}
\end{array}\right]
=
\tau
\left[\begin{array}{c|c}
\Theta_{\mathcal{B}_C}^{-1}+\rho I_l & 0 \\
\hline
0 & \Theta_{\mathcal{N}_C}^{-1}+\rho I_{n-l}
\end{array}\right]
\left[\begin{array}{c}
v_{\mathcal{B}_C} \\
\hline
v_{\mathcal{N}_C}
\end{array}\right]
\end{equation}
Obviously, eigenvalues of $P^{-1}H$ are all real, hence $\tau$ is also real. Multiplication of \eqref{eq:precH1} by $[v_{\mathcal{B}_C}\;;\;v_{\mathcal{N}_C}]^\mathsf{T}$ from the left gives
\begin{multline}\label{eq:tau1}
\tau\Bigl[v_{\mathcal{B}_C}^\mathsf{T}\left(\Theta_{\mathcal{B}_C}^{-1}+\rho I_l\right)v_{\mathcal{B}_C}
+
v_{\mathcal{N}_C}^\mathsf{T}\left(\Theta_{\mathcal{N}_C}^{-1}+\rho I_{n-l}\right)v_{\mathcal{N}_C}\Bigr] 
=\\
v_{\mathcal{B}_C}^\mathsf{T}\left(B^\mathsf{T}B-\rho I_l\right)v_{\mathcal{B}_C}
+
v_{\mathcal{N}_C}^\mathsf{T}\left(N^\mathsf{T}N-\rho I_{n-l}\right)v_{\mathcal{N}_C}
+
2v_{\mathcal{B}_C}^\mathsf{T}B^\mathsf{T}Nv_{\mathcal{N}_C}.
\end{multline}

Let us denote $\norm2{v_{\mathcal{B}_C}}^2$ by $\al$, then $\norm2{v_{\mathcal{N}_C}}^2=1-\al$ since $v=[v_{\mathcal{B}_C}\;;\;v_{\mathcal{N}_C}]$ has unit norm. Bounding left hand side of \eqref{eq:tau1} from below is trivial:
\begin{equation}\label{eq:boundLHS}
\Bigl|\tau\Bigl[v_{\mathcal{B}_C}^\mathsf{T}\left(\Theta_{\mathcal{B}_C}^{-1}+\rho I_l\right)v_{\mathcal{B}_C}
+
v_{\mathcal{N}_C}^\mathsf{T}\left(\Theta_{\mathcal{N}_C}^{-1}+\rho I_{n-l}\right)v_{\mathcal{N}_C}\Bigr]\Bigr|
\geq
|\tau|\Bigl(\rho\al
+
C(1-\al)\Bigr).
\end{equation}

Next, let us bound right hand side of \eqref{eq:tau1} from above. We will distinguish two cases, orthonormal and nearly orthonormal rows of matrix $A$. First, we study the case of nearly orthonormal rows of matrix $A$. For this purpose we will use the SVD decompositions of matrices $B$ and $N$:
\[
B=U_B\Sigma_{B}V_{B}^\mathsf{T},\quad
\Sigma_{B}=
\left[\begin{array}{c}
\diag(\sigma_1,\sigma_2,\ldots,\sigma_l) \\
\hline
O_{m-l\times l} \\
\end{array}\right]
\]
and 
\[
N = U_N\Sigma_{N}V_{N}^\mathsf{T},\quad
\Sigma_{N}=
\left[\begin{array}{c|c}
\diag(\varsigma_1,\varsigma_2,\ldots,\varsigma_{m}) & O_{m\times (n-m-l)}
\end{array}\right].
\]
Restricted isometry property \textbf{P2} implies that
\[
\sigma_1^2\leq\rho(1+\dt_l),\quad \sigma_l^2\geq\rho(1-\dt_l).
\]
First, notice that
\begin{equation}\label{eq:boundBTB}
\Bigl|v_{\mathcal{B}_C}^\mathsf{T}\left(B^\mathsf{T}B-\rho I_l\right)v_{\mathcal{B}_C}\Bigr|\leq\rho\dt_l\al.
\end{equation}
Using property \textbf{P1} we have
\begin{align}\label{bdsvdN}
\|AA^\mathsf{T} - I_m\| _2 & \le \delta & \Longleftrightarrow \nonumber \\
\|AA^\mathsf{T}\|_2            & \le 1+\delta & \Longleftrightarrow \nonumber \\
\|BB^\mathsf{T} + NN^\mathsf{T}\|_2 & \le 1+\delta & \Longrightarrow \nonumber \\
\|NN^\mathsf{T}\|_2 & \le 1+\delta &\Longleftrightarrow \nonumber \\
\varsigma_1^2 & \le 1+\delta. & 
\end{align}
Next, using (\ref{bdsvdN}) obtain
\[
\norm2{N^\mathsf{T}N-\rho I_{n-l}}\le\max\{\rho,1+\delta-\rho\}=1+\delta-\rho,\quad (\rho = m/n \leq 0.5)
\]
and, hence,
\begin{equation}\label{eq:boundNTN}
\Bigl|v_{\mathcal{N}_C}^\mathsf{T}\left(N^\mathsf{T}N-\rho I_{n-l}\right)v_{\mathcal{N}_C}\Bigr|\leq\Bigl(1+\delta-\rho\Bigr)\Bigl(1-\al\Bigr).
\end{equation}
Finally,
\[
\norm2{B^\mathsf{T}N}\le \|B\|_2\|N\|_2\le \sigma_1\varsigma_1=\sqrt{1+\delta}\sqrt{\rho(1+\dt_l)}< \sqrt{(1+\delta)}
\]
because our assumptions $m \leq n/2$ and $\dt_l<1$ imply $\sigma_i^2\leq\sigma_1^2\leq\rho(1+\dt_l)<1$. We conclude that
\begin{equation}\label{eq:boundBTN}
\Bigl|2v_{\mathcal{B}_C}^\mathsf{T}B^\mathsf{T}Nv_{\mathcal{N}_C}\Bigr|<2\sqrt{1+\delta}\sqrt{\al(1-\al)}.
\end{equation}

Bounds \eqref{eq:boundNTN} and \eqref{eq:boundBTN} are sharp and can be used to obtain very tight estimate on $\tau$ but we do not need them that sharp to obtain a sufficiently good estimate. So, we will release them a little bit to simplify the analysis:
\begin{equation}\label{eq:boundNTN_BTN}
\begin{aligned}
&\Bigl|v_{\mathcal{N}_C}^\mathsf{T}\left(N^\mathsf{T}N-\rho I_{n-l}\right)v_{\mathcal{N}_C}\Bigr|\le(1+\delta-\rho)(1-\al)\leq(1+\delta-\rho)\sqrt{1-\al},\\
&\Bigl|2v_{\mathcal{B}_C}^\mathsf{T}B^\mathsf{T}Nv_{\mathcal{N}_C}\Bigr|<2\sqrt{1+\delta}\sqrt{\al(1-\al)}\leq2\sqrt{1+\delta}\sqrt{1-\al}.
\end{aligned}
\end{equation}

Using \eqref{eq:boundLHS} and \eqref{eq:boundBTB} and \eqref{eq:boundNTN_BTN} we finally get
\begin{equation}\label{eq:tau2}
|\tau|\le \frac{\rho\dt_l\al+(1+\delta-\rho+2\sqrt{1+\delta})\sqrt{1-\al}}{\rho\al+C(1-\al)}\leq\dt_l(1+\varepsilon).
\end{equation}
Let us denote $\xi=(1+\delta-\rho+2\sqrt{1+\delta})$ and show that $\varepsilon$ is small for large values of $C$. Indeed \eqref{eq:tau2} implies that 
\[
\xi\sqrt{1-\al}\leq\dt_l\Bigl(C+C\varepsilon-\rho\varepsilon\Bigr)(1-\al)+\rho\dt_l\varepsilon.
\]
It can be checked by simple calculus, that $\sqrt{x}\leq C_1x+C_2$ on $[0,1]$ whenever $C_1\geq{1}/{(4C_2)}$. In our case this implies
\[
\frac{\dt_l}{\xi}\Bigl(C+C\varepsilon-\rho\varepsilon\Bigr)\geq
\frac{\xi}{4\rho\dt_l\varepsilon}.
\]
The largest solution of the quadratic equation in $\varepsilon$
\[
\frac{4\rho\dt_l^2}{\xi^2}\varepsilon\Bigl(C+C\varepsilon -\rho\epsilon \Bigr)=1
\]
is
\[
\varepsilon_{+}=\frac{1}{2}\cdot\frac{C}{C-\rho}\left(\sqrt{1+\frac{\xi^2}{\rho\dt_l^2 C}\cdot\frac{C-\rho}{C}}-1\right)\leq
\frac{\xi^2}{4\rho\dt_l^2 C}.
\]
Hence, it is sufficient to take any $\varepsilon\geq{\xi^2}/{(4\rho\dt_l^2 C)}$ to satisfy the inequality \eqref{eq:tau2}:
\begin{equation}\label{eq:clustering2}
\boxed{|\tau| \leq \dt_l+\frac{\xi^2}{4\rho\dt_l C}=   \dt_l+\frac{1}{4}\frac{(1+\delta-\rho+2\sqrt{1+\delta})^2}{\rho\dt_l C}}.
\end{equation}
This completes the proof for matrix $A$ which satisfies property \textbf{P1}. For the case of orthonormal rows of matrix $A$, i.e. $AA^\mathsf{T}=I_m$
simply set $\delta=0$ in property \textbf{P1} to get
\begin{equation}\label{eq:clustering4}
\boxed{|\tau| \leq \dt_l+\frac{1}{4}\frac{(3-\rho)^2}{\rho\dt_l C}}.
\end{equation}
This completes the proof.
\end{proof}

For the result of the theorem to be useful we obviously need the bound in the right-hand side of inequalities  in \eqref{eq:clustering2} and (\ref{eq:clustering4}) to be sufficiently smaller than one. Let us take a closer look at the terms forming this bound. We are free to choose any value for the constant $C$ we want, the larger the better. However, according to \eqref{eq:separator_C}, $l$ increases with the increase in $C$ and, consequently, the restricted isometry constant $\dt_l$ also increases. Inequalities \eqref{eq:clustering2} and (\ref{eq:clustering4}) hold for any value of $C$, hence we can replace it with
\begin{equation}\label{eq:clustering3}
|\tau|\leq \min_C\left(\dt_l+\frac{1}{4}\frac{(1+\delta-\rho+2\sqrt{1+\delta})^2}{\rho\dt_l C}\right)
\end{equation}
and
\begin{equation}\label{eq:clustering3}
|\tau|\leq \min_C\left(\dt_l+\frac{1}{4}\frac{(3-\rho)^2}{\rho\dt_l C}\right)
\end{equation}
and choose constant $C$ that delivers the minimum.

For number of measurements $m$ just a fraction $\rho=1/4$ of the length $n$ of the unknown signal, it is natural to assume the restricted isometry constant $\dt_{2l}$ to be less than $1/4$ (see Theorem \ref{thm:RIP_bound2}), hence, according to \cite{coraliaJarred}, $\dt_{2l}<1/4$, implies $\dt_l<1/4$. Therefore, to have $|\tau|\leq 17/20$ we need $C=20(0.75+\delta+2\sqrt{1+\delta})^2/3$ in \eqref{eq:separator_C}. For nearly orthonormal rows of matrix $A$ we can assume that $\delta\le1$, which gives us $C\approx 139.74$ and certainly holds near optimality in the IPM. For orthonormal rows of matrix $A$ we have $\delta=0$, hence, $C\approx 50.41$.

The bounds in \eqref{eq:clustering2} and (\ref{eq:clustering4}) are rather pessimistic. Computational experience suggests that eigenvalues of the preconditioned matrix get well clustered around 1 as long as $l=\#(\Theta_{j}^{-1}<1)$ is such that the RIP constant $\dt_l <1$. For example, for the discrete cosine (DCT) matrix with $n=2^{10}$ and $m=2^8$ the corresponding $l\leq 74$ (this number is obtained in a series of random tests).

Now we are ready to state the spectral properties of the preconditioned matrix $P^{-1}H$ for the system of equations \eqref{redNewtonA}. We leave the theorem without a proof as it a straightforward corollary of Lemma \ref{lem:ne2_prec}.

\begin{theorem}\label{thm:ne2_prec}
Let $H$ and $P$ be block matrices defined in \eqref{eq:ne2_H} and \eqref{eq:ne2_P}, respectively. 
Then the preconditioned matrix $P^{-1}H$ has
\begin{enumerate}
\item the eigenvalue $1$ of multiplicity $n$;
\item remaining $n$ eigenvalues defined in Lemma \ref{lem:ne2_prec} with $\Theta=\Theta_u+\Theta_v$.
\end{enumerate}
\end{theorem}

Theorem \ref{thm:ne2_prec} establishes the clustering of eigenvalues of $P^{-1}H$ around $1$. Hence, iterative method such as conjugate gradient applied to the system of equations \eqref{redNewtonA} is expected to converge in just a few iterations if the preconditioner $P$ in \eqref{eq:ne2_P} is used. 
The latter theoretical results are also confirmed in practical experiments. Figure \ref{figEigs} 
demonstrates clustering of eigenvalues $\lambda(H)$ and $\lambda(P^{-1}H)$ in the case that the $A$ matrix in $H$ (\ref{eq:ne2_H}) is a Discrete Cosine Transform (DCT) matrix with normalized rows, $AA^T=I$.
The parameters for the size of the problem are set to $m=2^{10}$, $n=2^{12}$ and the sparsity level is fixed to $k=51$. In the left sub-Figure \ref{fig1theorem} the clustering of the eigenvalues $\lambda(H)$
is shown. Every vertical line presents the spreading of $\lambda(H)$ at a particular CG call as the matrix-free IPM progresses. 
One can observe that the clustering worsens as
the matrix-free IPM approaches optimality. On the contrary, eigenvalues of the preconditioned matrices $P^{-1}H$ show the opposite behavior. In particular, as the matrix-free IPM progresses eigenvalues $\lambda(P^{-1}H)$
start to cluster around one. The latter is depicted with the vertical columns 
in the right sub-Figure \ref{fig2theorem}.
\begin{figure}%
\centering
\subfloat[Unpreconditioned systems, $H$]{%
\label{fig1theorem}%
\includegraphics[width=59mm,height=50mm]{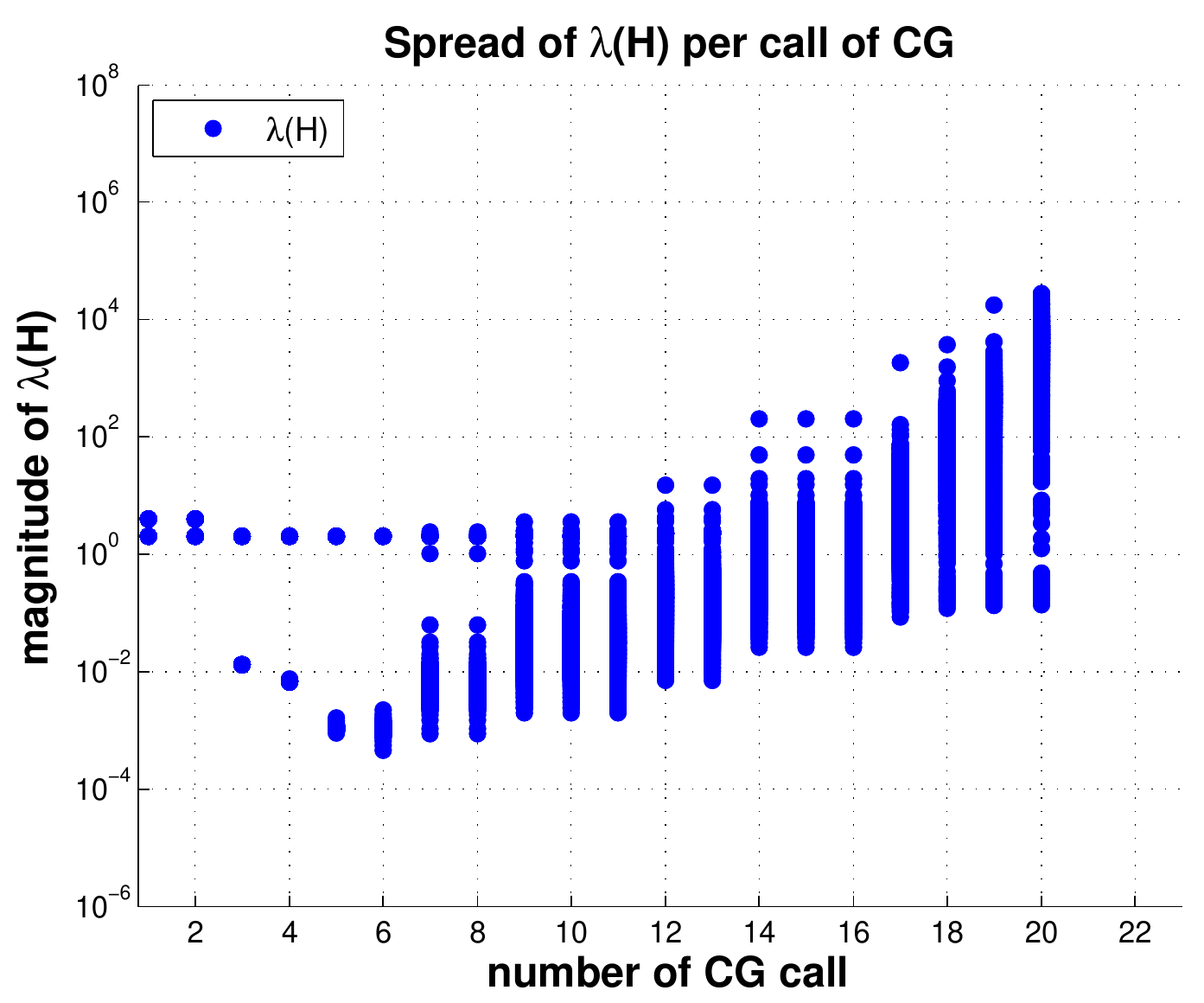}}
\subfloat[Preconditioned systems, $P^{-1}H$]{%
\label{fig2theorem}%
\includegraphics[width=59mm,height=50mm]{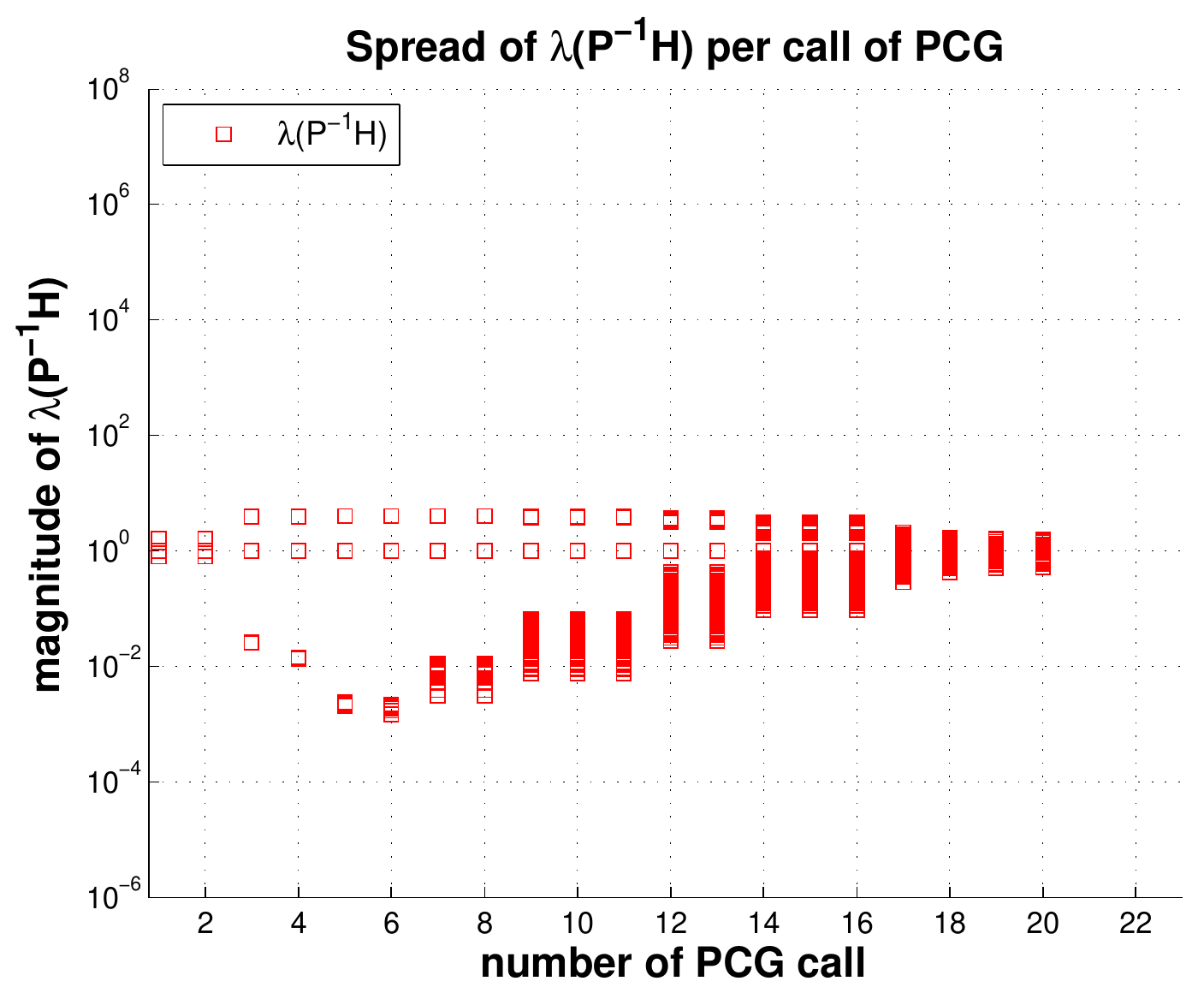}}\\
\caption{Clustering of the eigenvalues for the matrices $H$ and $P^{-1}H$
as the matrix-free IPM approaches optimality. The matrix $A$ in $H$ (\ref{eq:ne2_H}) is a DCT matrix with normalized rows.
The parameters of the problem set to $m=2^{10}$,$n=2^{12}$ and $k=51$. Twenty systems for the matrices
$H$ and $P^{-1}H$ are solved in total}
\label{figEigs}%
\end{figure}

\section{Computational Experience}\label{compexpsec}

We illustrate our developments by comparing the matrix-free 
IPM's efficiency with those of the state-of-the-art first-order methods, FPC\_AS and SPGL1
and with two other interior point based solvers, 
$\mathbf{\ell_1\_\ell_s}$ and PDCO.
The experiments are made on Sparco test suite \cite{IEEEhowto:sparco}.

We use the FPC\_AS CG version of FPC\_AS algorithm, where ``CG'' 
stands for the conjugate gradient method. The FPC\_AS CG has been shown 
in \cite{IEEEhowto:fpc1} to be considerably faster than other versions of FPC 
and FPC\_AS software packages. The FPC\_AS CG solves 
problem (\ref{formulationsnoisy1}). The code of FPC\_AS CG package can be found 
at \url{http://www.caam.rice.edu/~optimization/L1/FPC_AS/}.
We use the SPGL1\_bp version of SPGL1 
software package for noiseless signals and the SPGL1\_bpdn version for noisy 
signals, where ``bp'' stands for basis pursuit and ``bpdn'' for basis pursuit 
denoising, respectively. The SPGL1\_bp solves problem 
(\ref{onenormform}) and the ``bpdn'' version solves problem (\ref{formulationsnoisy3}).
The code of SPGL1 package can be found at \url{http://www.cs.ubc.ca/labs/scl/spgl1}.
Those versions of the FPC\_AS and SPGL1 software packages were found to be faster 
and more accurate than other first-order methods mentioned in Subsection~\ref{benchsubsec}. Therefore, GPSR and NestA solvers are excluded from the comparison.
The $\ell_1\_\ell_s$ solver implements problem (\ref{formulationsnoisy1}),
it can be found at \url{http://www.stanford.edu/~boyd/l1_ls/}.
The PDCO solver is used through the file SolveFasBP.m of SparseLab software 
package. The PDCO solver can be found at \url{http://www.stanford.edu/group/SOL/software/pdco.html}
and the SparseLab software package at \url{http://sparselab.stanford.edu/}.
The PDCO solver implements problems (\ref{onenormform}) and (\ref{formulationsnoisy1}).

In addition, three more experiments are performed. The first one tests the robustness of 
solvers matrix-free IPM, SPGL1\_bpdn,  FPC\_AS CG and $\ell_1\_\ell_s$, on problems of Sparco test suite, given a fixed level of noise. 
The second, replaces the core of matrix-free IPM, which is the preconditioned CG with a direct solver and shows how the CPU time required 
for reconstruction scales for each case. The third, 
demonstrates that the empirical phase transition properties of matrix-free IPM fit the theoretical average phase transition properties shown in \cite{DonohoJarred:PUT}.

All solvers used in this section, including the matrix-free IPM are MATLAB implementations. All experiments were performed
using MATLAB version R$2012$b ($8.0.0.783$) $64$-bit on a Dual $8$ Core Intel Xeon (Sandybridge) running Redhat Enterprise Linux in $64$-bit mode. Finally,
the RICE Wavelet toolbox, included in Sparco test suite, was compiled using gcc compiler version $4.4.6$ $20120305$ (Red Hat 4.4.6-4).
The matrix-free IPM, the data files and the MATLAB scripts used to generate the results in this section can be downloaded from \url{http://www.maths.ed.ac.uk/ERGO/mfipmcs/}.

Before proceeding to the following subsections it would be convenient for the reader to be 
familiarized with symbols and abbreviations used in the subsequent 
figures and comparison tables explained in Table~\ref{tablesymbabb}.
\begin{table}
\renewcommand{\arraystretch}{1.2}
\center
\caption{Symbols and abbreviations used in tables 
         and figures in Section "Computational Experience" \ref{compexpsec}}
	\begin{tabular}{|c|p{9cm}|}
	\hline
 		$m,n,k$&  number of rows and columns of the matrix A and the number of 
		nonzero elements in the optimal sparsest signal representation \\  \hline \hline
                $\hat{x}$ & optimal sparse representation \\  \hline \hline
                \multirow{2}{*}{$x_W$} & $x_{W_i}=x_i$ if $i\in W$,  \\
                               & otherwise $x_{W_i}=0$, where $W := \{i=1,2,\ldots,n \ | \ \hat{x}_i \neq 0 \}$ \\  \hline \hline
		\emph{r.e}($x_W$) & relative error $\|x_W-\hat{x}\|_2/\|\hat{x}\|_2$ \\  \hline \hline
		\emph{res}($x_W$) & residual $\|Ax_W-b\|_2$, where $b$ should be replaced with $\hat{b}$ in case of noiseless signals\\  \hline \hline
	         \emph{n1d}($x_W$) & distance from the optimal value of $\ell_1$-norm, $\Big|\|x_W\|_1 - \|\hat{x}\|_1\Big|$ \\  \hline \hline
	          \emph{obj}($x_W$) & objective value of BPDN problem, $\tau \|x_W\|_1 + \|Ax_W- b \|_2^2$ \\  \hline \hline
		\emph{nMat} & total number of matrix vector products $Ax$ and $A^\mathsf{T}y$\\ \hline
	\end{tabular}
	\label{tablesymbabb}
\end{table}

\subsection{Single centrality corrector primal-dual matrix-free IPM}\label{sec:pdstipm}
The implementation used in this paper is a single-corrector primal-dual IPM \cite{mybib:G-MCorr}. The original version proposed in \cite{mybib:G-MCorr} makes use of multiple centrality correctors, however,
after computational experimentation it was observed that a single corrector was enough for the fast convergence of the IPM in few iterations. In a standard multiple-corrector variant
at every iteration multiple centrality corrector directions are calculated, which are combined with a predictor direction in order to produce the final primal-dual direction  \cite{mybib:G-MCorr}. 
To compute the corrector and predictor directions one needs to solve multiple linear systems (\ref{redNewton}) where only the right
hand side varies. In case that a direct solver is used to solve the linear systems, the extra cost of solving several equations instead of one is negligible, because
the dominating cost is the decomposition of the matrix $(\Theta^{-1} + FF^\mathsf{T})$. However, this is not the case when iterative method (PCG) is used
to solve systems (\ref{redNewton}). In particular, the cost of calculating every term in composite direction is approximately the same.
In order to avoid the high cost of computing extra corrector directions at every iteration in our single-corrector matrix-free IPM we slightly bias the predictor direction to point to the central path
and perform corrector directions only when necessary. Like a long-step variant of primal-dual IPM \cite{IEEEhowto:wrightbook} this guarantees that at every iteration
the objective function is decreased rapidly while the algorithm maintains the small distance to the central path. 
As proposed in \cite{mybib:G-MCorr}, the criterion to decide whether a corrector direction is calculated is the value of the primal and dual step sizes.
When many biased predictor directions are performed the primal-dual iterates tend to approach the boundary of the feasible region. This results in small step sizes of the subsequent
iterations. When this happens a strong re-centering corrector is employed which pushes the next iteration to the vicinity of central path such that next step sizes are more likely to have large values.
Ideally, the values of the step sizes of the primal and dual directions should be bounded away from zero while global convergence of the method is guaranteed. 
This would allow fast practical convergence of matrix-free IPM, which translates into few iterations. Indeed, one can observe from the computational experience reported 
in Section~\ref{sec:cgcompdir} that $10$ to $20$ iterations of the matrix-free IPM is enough for convergence. This behaviour has been observed also in all computational
experiments discussed in Section~\ref{sec:comp}. We make our software available to the research community so that the interested reader can reproduce any numerical experiments from this paper. The pseudo-code of
the implemented single-corrector primal-dual matrix-free IPM follows.
\begin{algorithm}[H]
\begin{algorithmic}[1]
\vspace{0.1cm}
\STATE \textbf{Input} Choose $z^0, s^0 > 0$,  $0<\sigma_1<\sigma_2<\sigma_3\le 1$ and $0<\tilde{\alpha}<\bar{\alpha}<1$. For $k=1,2,\ldots$ generate $z^{k + 1}$ from $z^{k}$  and $s^{k+1}$ from $s^k$ according to the iteration: \vspace{0.25cm}
\WHILE{$\mbox{Duality Gap of (\ref{wrightBPDN}) and (\ref{dualquadraticBPDN})} \ge \epsilon$} \vspace{0.25cm}
\IF{$k \neq 1$ and ($\alpha_P^{k-1} \le \bar{\alpha}$ or $\alpha_D^{k-1} \le \bar{\alpha}$)}
\STATE $\sigma=\sigma_2$
\ELSE
\STATE $\sigma=\sigma_1$
\ENDIF \\ \vspace{0.25cm}
(* predictor step *)\vspace{0.1cm}
\STATE solve (\ref{redNewton}) using PCG with $\sigma$ and $z^k,s^k$ in (\ref{fzfs}) to obtain ($\Delta \bar{z}^k,\Delta \bar{s}^k$) \\\vspace{0.05cm}
	      choose primal and dual step sizes $\alpha_P^k,\alpha_D^k$ in $[0,1]$ as the largest values of $\alpha_P,\alpha_D$ such that 
               \begin{align*}
                 z^{k}(\alpha_P^k) = z^k + \alpha_P\Delta \bar{z}^k > 0 \\
                 s^{k}(\alpha_D^k) = s^k + \alpha_D\Delta \bar{s}^k > 0
               \end{align*} \\\vspace{0.25cm}
(* corrector step *)\vspace{0.1cm}
\IF{$\alpha_P^{k} \le \tilde{\alpha}$ or $\alpha_D^{k} \le \tilde{\alpha}$} \vspace{0.05cm}
\STATE  solve (\ref{redNewton}) using PCG with $\sigma=\sigma_3$ and $z^{k}(\alpha_P^k), s^{k}(\alpha_D^k)$ in (\ref{fzfs}) to obtain ($\Delta \tilde{z}^k,\Delta \tilde{s}^k$) \\\vspace{0.05cm}
               set ($\Delta z^{k},\Delta s^{k}$)=($\Delta \tilde{z}^k,\Delta \tilde{s}^k$)+($\Delta \bar{z}^k,\Delta \bar{s}^k$)\\\vspace{0.05cm}
	     choose primal and dual step sizes $\alpha_P^k,\alpha_D^k$ in $[0,1]$ as the largest values of $\alpha_P,\alpha_D$ such that 
               \begin{align*}
                 z^{k}(\alpha_P^k) = z^k + \alpha_P\Delta z^k > 0 \\
                 s^{k}(\alpha_D^k) = s^k + \alpha_D\Delta s^k > 0
               \end{align*} 
\ENDIF \vspace{0.25cm} 
\STATE set ($z^{k+1},s^{k+1}$)=($z^{k}(\alpha_P^k),s^{k}(\alpha_D^k)$)
\ENDWHILE
\end{algorithmic}
\caption*{Single-Corrector Primal-Dual Matrix-Free IPM}
\end{algorithm}
The input parameters $\sigma_1,\sigma_2$ are user-defined and control the centering bias of the predictor directions, while $\sigma_3$
parameter controls the strong centering in the corrector directions.
For all experiments they have been set to $\sigma_1=0.1$, $\sigma_2=0.5$ and $\sigma_3=0.8$. The input parameters $\bar{\alpha}$
and $\tilde{\alpha}$ are user-defined, $\bar{\alpha}$ controls whether $\sigma_1$ or $\sigma_2$ will be used as a centering parameter
for the predictor directions and $\tilde{\alpha}$ controls the frequency of the corrector updates. For all experiments they have been set 
to $\bar{\alpha}=0.5$ and $\tilde{\alpha}=0.1$.

\subsection{Benchmarks}\label{benchsubsec}
In order to have a base of comparison we choose to show the efficiency of the matrix-free
IPM on already existing benchmarks, which have been used by several 
researchers including \cite{IEEEhowto:spgl1,IEEEhowto:fpc1}. 
Experiments are performed on $18$ real valued sparse reconstruction problems, see Table \ref{sparcocollect}, from the Sparco collection \cite{IEEEhowto:sparco}. 
In total, Sparco collection consists of $26$ problems, out of which $6$ are complex valued and $20$ real valued. For the experiments in this section, 
the complex valued problems with IDs $1$, $4$, $8$, $501$ to $503$, are ignored since the matrix-free IPM manipulates only real data. Moreover, $2$ out of the $20$ 
real valued problems, with IDs $703$ and $901$, are also ignored because of their difficulty to be generated on any machine in a stand-alone approach, since, they require external packages such
as CurveLab \cite{curvelab} and FFTW \cite{fftw}. For problems
in Table \ref{sparcocollect} with IDs $401$ to $403$, $601$ to $603$, $701$ and $702$, the optimal representation $\hat{x}$ is not given by Sparco
toolbox. Therefore, the SPGL1\_bp solver is used to obtain $\hat{x}$ with required high accuracy. In particular to obtain $\hat{x}$, the parameters of SPGL1\_bp are set to
\begin{equation}\label{highacc}
\mbox{bpTol}  = \mbox{1.0e-15}, \quad \mbox{optTol}   = \mbox{1.0e-15}, \quad \mbox{decTol}  = 20\thinspace000.  
\end{equation}
where bpTol controls the tolerance for identifying a basis pursuit solution, optTol controls the optimality tolerance and decTol controls the frequency of Newton updates. 
Some of the components of the obtained solution from
the SPGL1\_bp might be nearly but not exactly zero, hence, as nonzero components are considered the ones in the set
$\mbox{\emph{nnz}}(x) := \{k=1,2,\ldots,n \ | \ \sum\nolimits_{i=1}^k|\bar{x}_i|\le 0.999\|x\|_1\}$, where $\bar{x}$ is the vector $x$ sorted in decreasing order of absolute values of its components. 
Then we set $\hat{x}_i=x_i$ if $i\in W$ otherwise $\hat{x}_i=0$, where $W := \{i=1,2,\ldots,n \ | \ i\in\mbox{ \emph{nnz}}(x) \}$.
%
\begin{table}
\renewcommand{\arraystretch}{1.2}
\caption{$18$ out of $20$ real valued problems of Sparco collection}
\label{sparcocollect}
\centering
\begin{tabular}{|l|c|c|l|c|}
\hline
\multicolumn{1}{|c|}{\multirow{1}{*}{Problem}} & ID & m,  n & \multicolumn{1}{c|}{\multirow{1}{*}{Operator}} &
\multicolumn{1}{c|}{\multirow{1}{*}{$\|\hat{x}\|_1$}} \\ \hline \hline
\cite{IEEEhowto:DonohoHuo,IEEEhowto:DonohoJohnstone} blocksig & 2 & 1\thinspace024,  1\thinspace024 & wavelet & 4.5e\scalebox{.75}{$\mplus$}02 \\  \hline \hline
\cite{IEEEhowto:DonohoHuo,IEEEhowto:DonohoJohnstone} blkheavi  & 9& 128,  128 & heaviside & 4.1e\scalebox{.75}{$\mplus$}01\\  \hline \hline
\cite{IEEEhowto:DonohoHuo,IEEEhowto:DonohoJohnstone}  blknheavi& 10 & 1\thinspace024,  1\thinspace024 & normal. heaviside & 9.8e\scalebox{.75}{$\mplus$}02\\  \hline \hline
\cite{IEEEhowto:gpsr} blurrycam & 701 &65\thinspace536,  65\thinspace536 & blurring, wavelet & 1.0e\scalebox{.75}{$\mplus$}04\\  \hline \hline
\cite{IEEEhowto:gpsr} blurspike & 702 &16\thinspace384,  16\thinspace384 & blurring & 3.4e\scalebox{.75}{$\mplus$}02\\  \hline \hline
cosspike & 3 & 1\thinspace024,  2\thinspace048 & DCT & 2.2e\scalebox{.75}{$\mplus$}02\\  \hline \hline
gausspike & 11 & 256,  1\thinspace024 & Gausian ens. & 2.4e\scalebox{.75}{$\mplus$}01\\  \hline \hline
gcosspike & 5 & 300,  2\thinspace048 & Gausian ens., DCT & 1.8e\scalebox{.75}{$\mplus$}02\\  \hline \hline
\cite{IEEEhowto:HennenfentHerrmanjitter} jitter & 902 &  200,  1\thinspace000 & DCT & 1.7e\scalebox{.75}{$\mplus$}00 \\  \hline \hline
\cite{IEEEhowto:CandesRomberg} p3poly & 6 & 600,  2\thinspace048 & Gausian ens., wavelet & 1.7e\scalebox{.75}{$\mplus$}03\\  \hline \hline
\cite{IEEEhowto:l1magic} sgnspike & 7 & 600,  2\thinspace560 & Gaussian ens. & 2.0e\scalebox{.75}{$\mplus$}01\\  \hline \hline
\cite{IEEEhowto:DossalMallat} spiketrn & 903 & 1\thinspace024,  1\thinspace024 & 1D convolution & 1.3e\scalebox{.75}{$\mplus$}01\\ \hline \hline
\cite{IEEEhowto:soccer1} soccer1 & 601 & 3\thinspace200, 4\thinspace096 & binary, wavelet & 4.2e\scalebox{.75}{$\mplus$}02 \\ \hline \hline
\cite{IEEEhowto:soccer1} soccer2 & 602 & 3\thinspace200, 4\thinspace096 & binary, Haar wavelet & 7.4e\scalebox{.75}{$\mplus$}02 \\ \hline \hline
srcsep1 & 401 & 29\thinspace166, 57\thinspace344 & windowed DCT & 1.0e\scalebox{.75}{$\mplus$}03 \\ \hline \hline
srcsep2 & 402 & 29\thinspace166, 86\thinspace016 & windowed DCT & 7.7e\scalebox{.75}{$\mplus$}03 \\ \hline \hline
srcsep3 & 403 & 196\thinspace608, 196\thinspace608 & blurring, wavelet & 1.0e\scalebox{.75}{$\mplus$}03 \\ \hline \hline
\cite{IEEEhowto:gpsr} yinyang & 603 & 1\thinspace024, 4\thinspace096 & wavelet & 2.6e\scalebox{.75}{$\mplus$}02 \\ \hline
\end{tabular}
\end{table}

Noise is introduced to the
noiseless measurements $\hat{b}$ using the following command in MATLAB:
\begin{equation}\label{corruptnois}
 b=\mbox{\texttt{awgn}}(\hat{b},\mbox{SNR},\mbox{`measured'}),
\end{equation} 
The function \texttt{awgn} is a MATLAB function 
from Communications Systems Toolbox which adds white Gaussian noise 
to signal $\hat{b}$. The SNR is the signal to noise ratio, measured in dB. 
The `measured' option specifies that the power of the signal is calculated
first before the addition of the noise.

\subsection{Equivalence of $\mbox{BP}_{\epsilon_2}$ and BPDN}\label{sec:equiv}
It has already being stated in Section~\ref{sec:intro} that problems $\mbox{BP}_{\epsilon_2}$ in (\ref{formulationsnoisy3}) and BPDN in (\ref{formulationsnoisy1})
are equivalent given particular parameters $\epsilon_2$ and $\tau$. In this paper the tested solvers implement
problem $\mbox{BP}_{\epsilon_2}$, i.e. SPGL1\_bpdn, or problem BPDN, i.e. matrix-free IPM, FPC\_AS CG $\ell_1\_\ell_s$
and PDCO. In order to perform a fair comparison among these solvers it has to be made certain that all codes solve equivalent problems.
Otherwise, different optimal solutions will be obtained, therefore, a straightforward and clear comparison would be impossible.
Unfortunately, exact values of $\epsilon_2$ and $\tau$ which make problems $\mbox{BP}_{\epsilon_2}$ and BPDN equivalent are not known a priori,
except for the case of orthogonal matrix $A$. However, given $\epsilon_2$ an approximate $\tau$ can be computed such that an approximate equivalence holds.

According to \cite{IEEEhowto:spgl1} given $\epsilon_2$ the parameter $\tau$ which makes problems $\mbox{BP}_{\epsilon_2}$ and BPDN equivalent,
is the optimal Lagrange multiplier of the dual problem of $\mbox{BP}_{\epsilon_2}$. Since, SPGL1\_bpdn outputs both the primal iterates and the optimal
Lagrange multiplier of $\mbox{BP}_{\epsilon_2}$, it can be used to approximately find $\tau$. Having such a parameter $\tau$ the BPDN solvers,
matrix-free IPM, FPC\_AS CG $\ell_1\_\ell_s$, PDCO and the $\mbox{BP}_{\epsilon_2}$ solver SPGL1\_bpdn can be legitimately compared.

Moreover, in order to be able to compare the quality of the reconstructed representations for each solver when solving equivalent problems, the optimal sparsest representation
for a particular level of noise needs to be known in advance. This is definitely not the case when noise is added manually by the user to a noiseless signal $\hat{b}$ using (\ref{corruptnois}). 
Due to manual corruption of signal $\hat{b}$, the energy of the added noise $\epsilon_2=\|e\|_2$ is known in advance. Hence, solving $\mbox{BP}_{\epsilon_2}$ will give the optimal
sparsest representation for this particular level of noise, $\epsilon_2$. This solution is obtained by first calling SPGL1\_bpdn solver to solve $\mbox{BP}_{\epsilon_2}$ by setting 
$\epsilon_2=\|e\|_2$ with required high accuracy, see (\ref{highacc}). During this process the approximate $\tau$ which makes problems $\mbox{BP}_{\epsilon_2}$ and BPDN equivalent
is obtained from SPGL1\_bpdn as has been described before. Hence, it is concluded that approximate $\tau$ and optimal sparse representations can be calculated
such that a fair comparison can be conducted.

Finally, for noiseless signals $\hat{b}$, the problem is easier. Problems BP (\ref{onenormform}) and BPDN \eqref{formulationsnoisy1} are almost equivalent for sufficiently small $\tau$, i.e. 1.0e-10. 
However, such a small $\tau$ can make the $\ell_1$-norm in BPDN numerically negligible, see Figure 6.2 in \cite{IEEEhowto:spgl1} for numerical examples.
For the former reason, if such a case is observed for BPDN solvers, parameter $\tau$ is set experimentally to a larger value, their values are given in Table \ref{regparams}.

\subsection{Termination Criteria and Parameter Tuning}\label{sec:tuning}
Termination of the compared solvers is forced when a solution of similar
quality to the one of matrix-free IPM is obtained. In order to do so, the termination
criteria of the compared solvers are changed. In particular, SPGL1 solver is
terminated when the following criteria are satisfied
$$
n1d(x_W^k) \le n1d(x_W^m), \quad r.e(x_W^k) \le r.e(x_W^m), \quad res=(x_W^k) \le res(x_W^m),
$$
where $x_W^k$ is the projected representation at the $k^{th}$ iteration of SPGL1 and $x_W^m$ is the projected representation obtained by matrix-free IPM.
Solvers FPC\_AS, $\ell_1\_\ell_s$ and PDCO are terminated when the following conditions are satisfied
$$
obj(x_W^k) \le obj(x_W^m), \quad r.e(x_W^k) \le r.e(x_W^m).
$$
Using these criteria for the compared solvers it is made certain that the reconstructed representations
have approximately the same $\ell_1$-norm, $\ell_2$-norm of residual $Ax_W-b$ and number of non
zero elements in $x_W$. The differentiation of the termination criteria for solver SPGL1 is done because SPGL1 solves problem $\mbox{BP}_{\epsilon_2}$, unlike
all other codes which solve the BPDN problem. Hence, it is more natural and fair for SPGL1 to be compared
with other solvers using termination criteria in SPGL1 way.

Occasionally, certain solvers required too
many matrix-vector products without achieving a solution of similar quality
to the one delivered by the matrix-free IPM. In this case the solvers were terminated
when $\mbox{\emph{nMat}} > 40\thinspace000$. 

Regarding the parameter tuning of the compared solvers, all their parameters are set
to their default values. For the matrix-free IPM the following parameters need to be set.
\begin{itemize}
\item tol: Relative duality gap of primal-dual pair (\ref{wrightBPDN}) and (\ref{dualquadraticBPDN}). For noisy problems, this parameter varies
between 1.0e-6 and 1.0e-10. For noiseless problems it varies between 1.0e-7 and 1.0e-14.
\item maxiters: Maximum number of iterations. For all problems this parameter is set to $100$.
\item tolpcg: Tolerance of preconditioned CG method. For noisy problems this parameter varies between 1.0e-1 and 1.0e-2 and
for noiseless ones it varies between 1.0e-1 and 1.0e-6.
\item mxiterpcg: Maximum number of iterations of preconditioned CG method. For all problems this parameter is set
to $200$.
\end{itemize}
Since a large number of experiments has been performed, the exact parameter tuning
of matrix-free IPM is not given here. However, it can be found in the MATLAB scripts which reproduce
the results in this section, see \url{http://www.maths.ed.ac.uk/ERGO/mfipmcs/}.

Finally, the parameters $\epsilon_2$ and $\tau$ in problems (\ref{formulationsnoisy3}) and (\ref{formulationsnoisy1}), respectively,
for noisy problems are set as described in Subsection~\ref{sec:equiv} for $\epsilon_2=\|e\|_2$. For noiseless problems $\tau$ is set to 
arbitrarily small values given in Table \ref{regparams}.
\subsection{Comparison}\label{sec:comp}
In this section we present computational results 
obtained for the Sparco collection problems discussed in the Benchmarks section. 
Both noisy and noiseless measurements are considered. 
Noise is added to measurements using (\ref{corruptnois}) by fixing
the $\mbox{SNR}=60$ dB.
A comparison among the previously mentioned solvers is made in terms of the quality of 
reconstruction and computational effort. The results of experiments are shown in Table \ref{table4}.
The first column in Table \ref{table4} shows the IDs of the Sparco problems. For each ID the first and second sub-rows give results for noisy and noiseless measurements, respectively. 
The second column reports the $\ell_1$-norm of the projected reconstructed representation for matrix-free IPM. 
The third column shows the relative error $r.e$, see Table \ref{tablesymbabb}, of the projected reconstructed representation that was achieved by matrix-free IPM. 
The forth column shows the $\ell_2$-norm of the residual, denoted by $res$ in Table \ref{tablesymbabb}, for matrix-free IPM.
The rest of the table shows the number of matrix-vector 
products, \emph{nMat}, that were needed by each solver to reconstruct a solution of similar quality to the one of
matrix-free IPM. In cases when number of matrix-vector products required by a solver exceeded $40\thinspace000$, 
the solver was terminated with a failure status. To be precise, it is a failure to 
converge to a solution similar to the one obtained by matrix-free IPM.
Problems for which the matrix-free IPM converged with the lowest number of matrix-vector products among all solvers compared are denoted in bold.
In Table \ref{regparams} are shown the regularization parameters $\tau$ for noiseless signals that were used for BPDN for solvers
matrix-free IPM,  FPC\_AS, $\ell_1\_\ell_s$ and PDCO. Finally, for noiseless signals the version SPGL1\_bp of SPGL1 solver
is called. 

One can observe in Table \ref{table4} that the matrix-free IPM was the fastest
solver in 11 out of 36 noisy and noiseless problems, while it was the second fastest for another 14 problems,
denoted by italic font.
\begin{table}
\center
\caption{Results for noisy and noiseless Sparco problems}
\begin{tabular}{|c|c|c|c|r|r|r|r|r|}
\hline
\multicolumn{4}{|c|}{}&\multicolumn{1}{c|}{mfIPM}&\multicolumn{1}{c}{$\ell_1\_\ell_s$}&\multicolumn{1}{|c|}{PDCO}
  &\multicolumn{1}{c|}{FPC\_AS}&\multicolumn{1}{c|}{SPGL1} \\ \hline
ID & $\|x\|_1$ & \emph{r.e} & \emph{res} &\multicolumn{5}{c|}{\emph{nMat}} \\ \hline \hline

\multirow{2}{*}{2}   & 4.5e\scalebox{.75}{$\mplus$}02 &5.3e-04& 8.2e-02 &\emph{61}         &726    &6\thinspace611   &9          &40\thinspace000\\
                     	      & 4.5e\scalebox{.75}{$\mplus$}02 &1.0e-11 &8.4e-10 &\emph{65}         &644    &40\thinspace011 &40\thinspace002 &21\\ \hline \hline

\multirow{2}{*}{3}   & 2.2e\scalebox{.75}{$\mplus$}02 &9.9e-04 &1.3e-01 &195        &446   &5\thinspace115     &119   &70\\
                                  & 2.2e\scalebox{.75}{$\mplus$}02 &1.8e-08 &1.8e-06 &387        &1\thinspace540 &40\thinspace005  &192   &146\\ \hline \hline

\multirow{2}{*}{5}   & 1.8e\scalebox{.75}{$\mplus$}02 &3.0e-03 &2.3e-01&1\thinspace367       &5\thinspace042    &28\thinspace369 &630   &510\\
                                  & 1.8e\scalebox{.75}{$\mplus$}02 &2.4e-05 &1.8e-03&\emph{6\thinspace239}       &20\thinspace758 &41\thinspace479 &636   &40\thinspace000\\ \hline \hline
                                  
\multirow{2}{*}{6}   & 1.7e\scalebox{.75}{$\mplus$}03 &2.2e-02 &1.7e\scalebox{.75}{$\mplus$}01&\emph{2\thinspace507}        &2\thinspace838   &42\thinspace125   &720    &40\thinspace000\\
                                  & 1.7e\scalebox{.75}{$\mplus$}03 &4.6e-02 &3.6e\scalebox{.75}{$\mplus$}01&\emph{7\thinspace193}        &40\thinspace011   &18\thinspace685   &573    &40\thinspace000\\ \hline \hline

\multirow{2}{*}{7}   & 2.0e\scalebox{.75}{$\mplus$}01 &5.4e-04 &2.3e-03&165        &452   &955   &78    &63\\
                                  & 2.0e\scalebox{.75}{$\mplus$}01 &5.6e-07 &1.1e-06&259        &952   &709   &78    &87\\ \hline \hline

\multirow{2}{*}{9}   & 4.1e\scalebox{.75}{$\mplus$}01 &1.0e-03 &1.7e-01&\bf{377}    &574  &579 &446    &8\thinspace855\\
                                  & 4.1e\scalebox{.75}{$\mplus$}01 &5.2e-12 &1.6e-10&\bf{661}   &3\thinspace860  &7\thinspace113  &40\thinspace002 &40\thinspace000\\ \hline \hline

\multirow{2}{*}{10}  & 9.0e\scalebox{.75}{$\mplus$}02 &9.3e-02 &3.3e\scalebox{.75}{$\mplus$}00&\emph{2\thinspace431}  &11\thinspace421  &1\thinspace043  &40\thinspace001 &40\thinspace000\\
                                  & 9.8e\scalebox{.75}{$\mplus$}02 &1.0e-09 &8.9e-08                &\bf{4\thinspace519}   &8\thinspace192  &42\thinspace647 &40\thinspace001 &40\thinspace000\\ \hline \hline
                                  
\multirow{2}{*}{11}  & 2.4e\scalebox{.75}{$\mplus$}01 &1.4e-03 &1.3e-01                &767  &2\thinspace186  &3\thinspace291  &217 &143\\
                                  & 2.4e\scalebox{.75}{$\mplus$}01 &6.8e-05 &5.2e-03                &1\thinspace241&4\thinspace542  &4\thinspace299 &219 &189\\ \hline \hline
                                  
 \multirow{2}{*}{401}& 1.0e\scalebox{.75}{$\mplus$}03 &8.9e-02 &1.2e-01                &\emph{2\thinspace747}  &42\thinspace622  &61\thinspace327  &40\thinspace076 &882\\
                                    & 1.0e\scalebox{.75}{$\mplus$}03 &7.7e-02 &9.7e-02                &\emph{3\thinspace193}   &43\thinspace512    &48\thinspace511    &40\thinspace076 &814\\ \hline \hline

 \multirow{2}{*}{402}& 1.0e\scalebox{.75}{$\mplus$}03 &1.0e-01 &1.9e-01                &\emph{4\thinspace393}  &46\thinspace458  &44\thinspace169  &40\thinspace078 &517\\
                                    & 1.0e\scalebox{.75}{$\mplus$}03 &8.1e-02 &2.0e-01                &\emph{4\thinspace991} &49\thinspace122    &43\thinspace845    &40\thinspace078 &617\\ \hline \hline
                                    
  \multirow{2}{*}{403}& 7.6e\scalebox{.75}{$\mplus$}03 &1.2e-02 &7.1e-01                &{2\thinspace841}  &6\thinspace136  &40\thinspace495  &2\thinspace305 &699\\
                                    & 7.7e\scalebox{.75}{$\mplus$}03 &4.1e-03 &9.2e-02                &\emph{6\thinspace031}&43\thinspace278    &69\thinspace913    &40\thinspace046 &932\\ \hline \hline

 \multirow{2}{*}{601}& 3.3e\scalebox{.75}{$\mplus$}02 &6.1e-02 &5.7e\scalebox{.75}{$\mplus$}01       &\bf{1\thinspace179}  &14\thinspace684  &40\thinspace153  &40\thinspace080 &40\thinspace000\\
                                    & 4.0e\scalebox{.75}{$\mplus$}02 &3.9e-02 &4.8e\scalebox{.75}{$\mplus$}00       &\emph{4\thinspace409}  &9\thinspace664    &43\thinspace369    &40\thinspace076 &1\thinspace116\\ \hline \hline
                                    
 \multirow{2}{*}{602}& 5.9e\scalebox{.75}{$\mplus$}02 &1.0e-01 &4.8e\scalebox{.75}{$\mplus$}01       &\emph{1\thinspace199}  &17\thinspace097  &40\thinspace631  &40\thinspace023 &898\\
                                    & 6.4e\scalebox{.75}{$\mplus$}02 &1.1e-01 &3.2e\scalebox{.75}{$\mplus$}00       &\bf{4\thinspace669}  &22\thinspace392    &42\thinspace139    &40\thinspace043 &40\thinspace000\\ \hline \hline

 \multirow{2}{*}{603}& 2.6e\scalebox{.75}{$\mplus$}02 &4.1e-03 &4.2e-02       &\emph{1\thinspace777}  &40\thinspace693  &50\thinspace369  &40\thinspace002 &443\\
                                    & 2.5e\scalebox{.75}{$\mplus$}02 &4.6e-02 &5.9e-01       &3\thinspace545  &2\thinspace350    &40\thinspace181    &338 &95\\ \hline \hline
                                                                                                                                                
 \multirow{2}{*}{701}& 9.1e\scalebox{.75}{$\mplus$}03 &4.6e-02 &1.5e-01       &\bf{1\thinspace217}  &33\thinspace160  &91\thinspace147  &40\thinspace044 &1\thinspace658\\
                                    & 1.0e\scalebox{.75}{$\mplus$}04 &2.4e-07 &4.1e-03       &\bf{1\thinspace907}  &4\thinspace722    &49\thinspace093    &40\thinspace001 &40000\\ \hline \hline

 \multirow{2}{*}{702}& 3.4e\scalebox{.75}{$\mplus$}02 &4.8e-03 &3.4e-03       &\bf{711}  &1\thinspace600  &5\thinspace525  &40\thinspace001 &40\thinspace000\\
                                    & 3.4e\scalebox{.75}{$\mplus$}02 &6.4e-08 &2.4e-03       &\bf{1\thinspace913}  &3\thinspace030    &49\thinspace009    &40\thinspace037 &12\thinspace388\\ \hline \hline

 \multirow{2}{*}{902}& 1.7e\scalebox{.75}{$\mplus$}00 &5.3e-04 &5.2e-04       &143  &498  &237  &40 &49\\
                                    & 1.7e\scalebox{.75}{$\mplus$}00 &2.0e-06 &9.6e-07       &239  &675    &279    &42 &59\\ \hline \hline

 \multirow{2}{*}{903}& 1.3e\scalebox{.75}{$\mplus$}01 &2.4e-03 &1.4e-01       &\bf{3\thinspace105}  &8\thinspace466  &4\thinspace775  &8\thinspace237 &6\thinspace735\\
                                    & 1.3e\scalebox{.75}{$\mplus$}01 &3.5e-06 &1.9e-04       &\bf{4\thinspace163} &25\thinspace128    &30\thinspace979    &33\thinspace529 &40\thinspace000\\  
\hline
\end{tabular}
\label{table4}
\end{table}
\begin{table}
\renewcommand{\arraystretch}{1.2}
\center
\caption{Regularization parameters $\tau$ for problem BPDN and noiseless measurements $\hat{b}$ for the experiments performed in Table \ref{table4}}
	\begin{tabular}{|c|c|}
	\hline
	         $\tau$ & Problems \\ \hline \hline
 		1.0e-10 &  $2$, $9$, $10$, $701$, $702$ \\  \hline \hline
                   1.0e-08  &  $401$, $402$, $603$ \\  \hline \hline
                   1.0e-07 & $3$, $7$, $902$ \\  \hline \hline
                   1.0e-05 & $903$\\  \hline \hline
                   1.0e-04 & $5$, $403$, $601$, $602$\\  \hline \hline
                   1.0e-03 & $6$\\  \hline \hline
                   1.0e-02 & $11$ \\ \hline
	\end{tabular}
	\label{regparams}
\end{table}
It is important to be mentioned that the performance of the compared solvers crucially depends on the condition number of matrices
build of subsets of columns of matrix $A$ with cardinality $\kappa$, less than $m$, i.e. full-rank sub-matrices of $A$. Unfortunately, it is a computational demanding task to 
check the condition number of every full-rank sub-matrix for the problems shown in Table \ref{sparcocollect}. Nevertheless, by experimenting with a few sub-matrices
one can get a picture of how well-conditioned sub-matrices of $A$ might be.

Based on the previous criterion we observed that on problems that the matrix-free IPM was first or second, matrix $A$ had relatively ill-conditioned sub-matrices,
at least for the ones that we experimented with. The previous implies that the proposed preconditioner was not as efficient
as predicted in Section \ref{sec:PCG}. However, the ill-conditioning also adversely affected the performance of SPGL1 and FPC\_AS, as
shown in Table \ref{table4}.
On the contrary, on problems that matrix $A$ seemed to have well-conditioned sub-matrices, the preconditioner was very efficient,
which resulted in a very fast matrix-free IPM. However, SPGL1 and FPC\_AS were faster. For example, see problems with IDs $2$, $3$, $7$ and $902$.

\subsection{Robustness to Noise}\label{sec:robust}
In this subsection we compare the matrix-free IPM with SPGL1, FPC\_AS CG, $\ell_1\_\ell_s$, in terms of their reconstruction capabilities 
for different levels of noise. The results collected in Table \ref{table4} and analysed in Section~\ref{sec:comp} reveal that PDCO and $\ell_1\_\ell_s$
demonstrate comparable efficiency but the latter is usually faster. Therefore, solver PDCO will not be used in our further experiment.

For this experiment, the level of noise is varied from $\mbox{SNR}=10$ dB to $\mbox{SNR}=120$ dB with a step of $10$ dB. The quality of reconstruction
for all solvers is measured using the amplitude criterion \cite{thomsonsingl}
\begin{equation*}
amp(x_W) = \frac{\sqrt{\frac{1}{n}\|x_W-\hat{x}\|_2^2}}{\sqrt{\frac{1}{m}\|e\|_2^2}}.
\end{equation*}
The main purpose of using the $amp$ criterion, instead of $r.e$, is that the former
amplifies the $r.e$, the nominator of $amp$, as $\|e\|_2\to 0$. Hence, less accurate
representations will be emphasized.

As in Section~\ref{sec:comp} when the optimal representation $\hat{x}$ of BP is unknown it is calculated approximately using solver
SPGL1\_bp with required high accuracy (\ref{highacc}). In order to have a fair comparison it is necessary to know at least approximately 
the parameter $\tau$ which makes problems $\mbox{BP}_{\epsilon_2}$ and BPDN equivalent and moreover, the optimal sparse
representation of $\mbox{BP}_{\epsilon_2}$ for $\epsilon_2=\|e\|_2$. The former issues are solved as described in Subsection~\ref{sec:equiv}.

To compare the solvers the following criterion is defined
\begin{equation}\label{ampcomp}
rampd(x_W)=\frac{\displaystyle\max(amp(x_W^*)-amp(x_W^s),0)}{amp(x_W^s)},
\end{equation}
where $rampd$ stands for relative amplitude difference, $x_W^*$ is the reconstructed projected representation by solvers matrix-free IPM, FPC\_AS CG, $\ell_1\_\ell_s$, and $x_W^s$ is the
reconstructed projected representation of solver SPGL1\_bpdn. Notice that if $rampd$ equals zero, then the representation $x_W^*$ is of better quality than
$x_W^s$, otherwise the inverse is true. 

In Table \ref{tablerobnoise1} is shown the average value of $rampd$ over all SNR for each solver. 
The first column of Table \ref{tablerobnoise1} reports the ID of every Sparco problem. From the second to the forth column the average $rampd$ over all SNR for each solver is shown.
The last three columns report the average $rampd$ for SNRs from $10$ dB to $60$ dB for each solver. Notice in Table \ref{tablerobnoise1} that matrix-free IPM for problems
with IDs $2$ to $11$ and $701$ to $903$ was consistently recovering a high quality solution. For problems with IDs $401$ to $603$ for $\mbox{SNR}>60$ dB all BPDN solvers, matrix-free IPM, FPC\_AS CG
and $\ell_1\_\ell_s$, were unable to reconstruct an adequate representation and this is in contrast to SPGL1. A similar observation has been reported in \cite{IEEEhowto:spgl1}. In this work the authors mentioned that this issue of BPDN
solvers might be due to very small regularisation parameter $\tau$, obtained from SPGL1 solver as the energy of noise is decreased. In this case, the regularization effect of the $\ell_1$-norm starts to be negligible and the solvers face considerable numerical difficulties.
However, in our experiments we observed for these problems that not always the $\tau$ parameter was small and additionally, there were other problems were $\tau$ was even smaller but successful reconstruction was possible. 
Therefore, we conclude that this failure of BPDN solvers might be 
problem dependent.
\begin{table}
\center
\caption{Average quality reconstruction results over SNR from $10$ dB to $120$ dB for solvers matrix-free IPM, FPC\_AS and $\ell_1\_\ell_s$ on Sparco problems in Table \ref{sparcocollect}}
\begin{tabular}{|c|x{1.39cm}|x{1.39cm}|x{1.39cm}|x{1.39cm}|x{1.39cm}|x{1.39cm}|}
\hline
     & \multicolumn{3}{x{4.17cm}|}{Avg. $rampd$ for SNR from $10$ dB to $120$ dB} & \multicolumn{3}{x{4.17cm}|}{Avg. $rampd$ for SNR from $10$ dB to $60$ dB}  \\ \hline
ID & mfIPM &FPC\_AS &$\ell_1\_\ell_s$ & mfIPM &FPC\_AS &$\ell_1\_\ell_s$ \\ \hline \hline
$  2$ & 6.1e-09 & 2.5e-10 & 0.0e\scalebox{.75}{$\mplus$}00 & 2.8e-13 & 2.6e-13 & 0.0e\scalebox{.75}{$\mplus$}00 \\ \hline \hline 
$  3$ & 1.3e-04 & 8.7e-05 & 0.0e\scalebox{.75}{$\mplus$}00 & 4.0e-09 & 6.4e-14 & 0.0e\scalebox{.75}{$\mplus$}00 \\ \hline \hline 
$  5$ & 5.1e-06 & 5.0e-07 & 1.7e-01 & 7.1e-11 & 9.8e-07 & 0.0e\scalebox{.75}{$\mplus$}00 \\ \hline \hline 
$  6$ & 1.2e-07 & 2.4e-10 & 1.1e\scalebox{.75}{$\mplus$}00 & 2.5e-07 & 4.8e-10 & 0.0e\scalebox{.75}{$\mplus$}00 \\ \hline \hline 
$  7$ & 1.5e-02 & 5.7e-08 & 0.0e\scalebox{.75}{$\mplus$}00 & 4.3e-06 & 3.5e-15 & 0.0e\scalebox{.75}{$\mplus$}00 \\ \hline \hline 
$  9$ & 1.1e-08 & 1.1e-01 & 4.9e-06 & 2.1e-08 & 2.1e-01 & 1.0e-08 \\ \hline \hline 
$ 10$ & 7.3e-04 & 1.6e-01 & 0.0e\scalebox{.75}{$\mplus$}00 & 1.5e-03 & 2.5e-01 & 0.0e\scalebox{.75}{$\mplus$}00 \\ \hline \hline 
$ 11$ & 4.2e-05 & 1.8e-05 & 0.0e\scalebox{.75}{$\mplus$}00 & 1.4e-10 & 3.7e-12 & 0.0e\scalebox{.75}{$\mplus$}00 \\ \hline \hline 
$401$ & 8.6e\scalebox{.75}{$\mplus$}00 & 1.2e\scalebox{.75}{$\mplus$}01 & 8.5e\scalebox{.75}{$\mplus$}00 & 1.8e-01 & 1.9e-01 & 1.7e-01 \\ \hline \hline 
$402$ & 8.0e\scalebox{.75}{$\mplus$}00 & 2.2e\scalebox{.75}{$\mplus$}01 & 8.0e\scalebox{.75}{$\mplus$}00 & 2.0e-01 & 1.9e\scalebox{.75}{$\mplus$}01 & 2.1e-01 \\ \hline \hline 
$403$ & 1.9e\scalebox{.75}{$\mplus$}00 & 3.8e\scalebox{.75}{$\mplus$}00 & 1.2e\scalebox{.75}{$\mplus$}00 & 8.1e-03 & 6.4e-12 & 1.4e-02 \\ \hline \hline 
$601$ & 3.8e\scalebox{.75}{$\mplus$}05 & 4.0e\scalebox{.75}{$\mplus$}03 & 1.5e\scalebox{.75}{$\mplus$}01 & 4.8e-11 & 8.1e\scalebox{.75}{$\mplus$}03 & 1.9e-01 \\ \hline \hline 
$602$ & 1.6e\scalebox{.75}{$\mplus$}00 & 2.9e\scalebox{.75}{$\mplus$}03 & 7.4e\scalebox{.75}{$\mplus$}00 & 1.1e-10 & 5.7e\scalebox{.75}{$\mplus$}03 & 1.3e-01 \\ \hline \hline 
$603$ & 8.1e-01 & 6.7e\scalebox{.75}{$\mplus$}00 & 7.7e-01 & 2.2e-08 & 3.7e-01 & 1.8e-03 \\ \hline \hline 
$701$ & 6.4e-08 & 1.9e\scalebox{.75}{$\mplus$}00 & 3.2e-03 & 0.0e\scalebox{.75}{$\mplus$}00 & 3.8e\scalebox{.75}{$\mplus$}00 & 6.4e-03 \\ \hline \hline 
$702$ & 7.9e-02 & 2.5e\scalebox{.75}{$\mplus$}01 & 6.2e-03 & 0.0e\scalebox{.75}{$\mplus$}00 & 5.1e\scalebox{.75}{$\mplus$}01 & 1.3e-03 \\ \hline \hline 
$902$ & 9.1e-02 & 9.7e-09 & 0.0e\scalebox{.75}{$\mplus$}00 & 2.1e-07 & 1.3e-08 & 0.0e\scalebox{.75}{$\mplus$}00 \\ \hline \hline 
$903$ & 1.5e-04 & 3.8e\scalebox{.75}{$\mplus$}00 & 1.5e-04 & 1.0e-11 & 7.5e\scalebox{.75}{$\mplus$}00 & 0.0e\scalebox{.75}{$\mplus$}00 \\  
\hline
\end{tabular}
\label{tablerobnoise1}
\end{table}

\subsection{Preconditioned Conjugate Gradient Method Against Direct Linear Solver}\label{sec:cgcompdir}
In this subsection we replace PCG in steps $8$ and $10$ of matrix-free IPM with a direct linear solver.
It has been mentioned in Section~\ref{sec:intro} that direct linear solvers are efficient when the system
to be solved is sufficiently sparse. However, for CS the systems (\ref{redNewton}) to be solved are completely dense
due to density of matrix $A$. For this reason, large scale problems
are not storable in a moderate computer with $8$ giga byte of random access memory. 
Even worse, matrix $A$ might be an algorithmic operator, i.e. DCT, therefore, direct solvers cannot be employed. Hence, direct linear solvers
for CS inside an IPM are only applicable when the measurement matrix $A$ is explicitly available, i.e. Gaussian matrix, and only for small scale problems, i.e. $n=2^{12}$ or smaller.
In addition to the former disadvantages of a direct solver for CS problems, its computational complexity for systems (\ref{redNewton}) will
be of order $\mathcal{O}(n^3)$. This is a well known result, for completely dense linear systems. Therefore, it is expected that for very small
instances the two approaches might require similar CPU time to converge, while as dimensions grow the CPU time of the IPM version with the direct linear solver
will increase rapidly. Indeed, this is confirmed by Figure \ref{fig1cpu}. 

Despite the higher computational effort required by direct solvers for CS problems, such an approach will produce exact Newton directions, hence,
one would expect that IPM iterations will be the minimum possible. Suprisingly, in Figure \ref{fig2iters} we show, that matrix-free IPM with PCG requires as few iterations
as its IPM version with a direct linear solver. Indeed, recent analysis of \cite{jacekinexact} indicates that allowing the use of inexact Newton directions in an IPM does not
adversely affect the worst-case complexity result of this method.

In the experiments reported in Figures  \ref{fig1cpu} and  \ref{fig2iters} matrix $A$ is Gaussian, the sparsity pattern of the optimal representation $\hat{x}$ 
is chosen at random, while the nonzero components follow a standard normal distribution. The noiseless measurements are produced by $\hat{b}=A\hat{x}$. The size of 
problem $n$, is varied from $2^5$ to $2^{12}$ with a step of times $2$, the measurements $m$ are varied from $2^3$ to $2^{10}$ with a step of times $2$ and 
the sparsity level $k$ is set to $\lceil m/{20}\rceil$. Finally, the $\tau$ parameter in BPDN problem (\ref{formulationsnoisy1}) is set to $\tau=$1.0e-3.
To solve the linear systems we use the \texttt{mldivide} function of MATLAB, which in case of symmetric real matrices with positive diagonal, i.e. (\ref{redNewton}), 
performs Cholesky factorisation. For details of the \texttt{mldivide} function we refer the reader to \url{http://www.mathworks.co.uk/help/matlab/math/systems-of-linear-equations.html}.
\begin{figure}%
\centering
\subfloat[Scaling of CPU time]{%
\label{fig1cpu}%
\includegraphics[width=59mm,height=50mm]{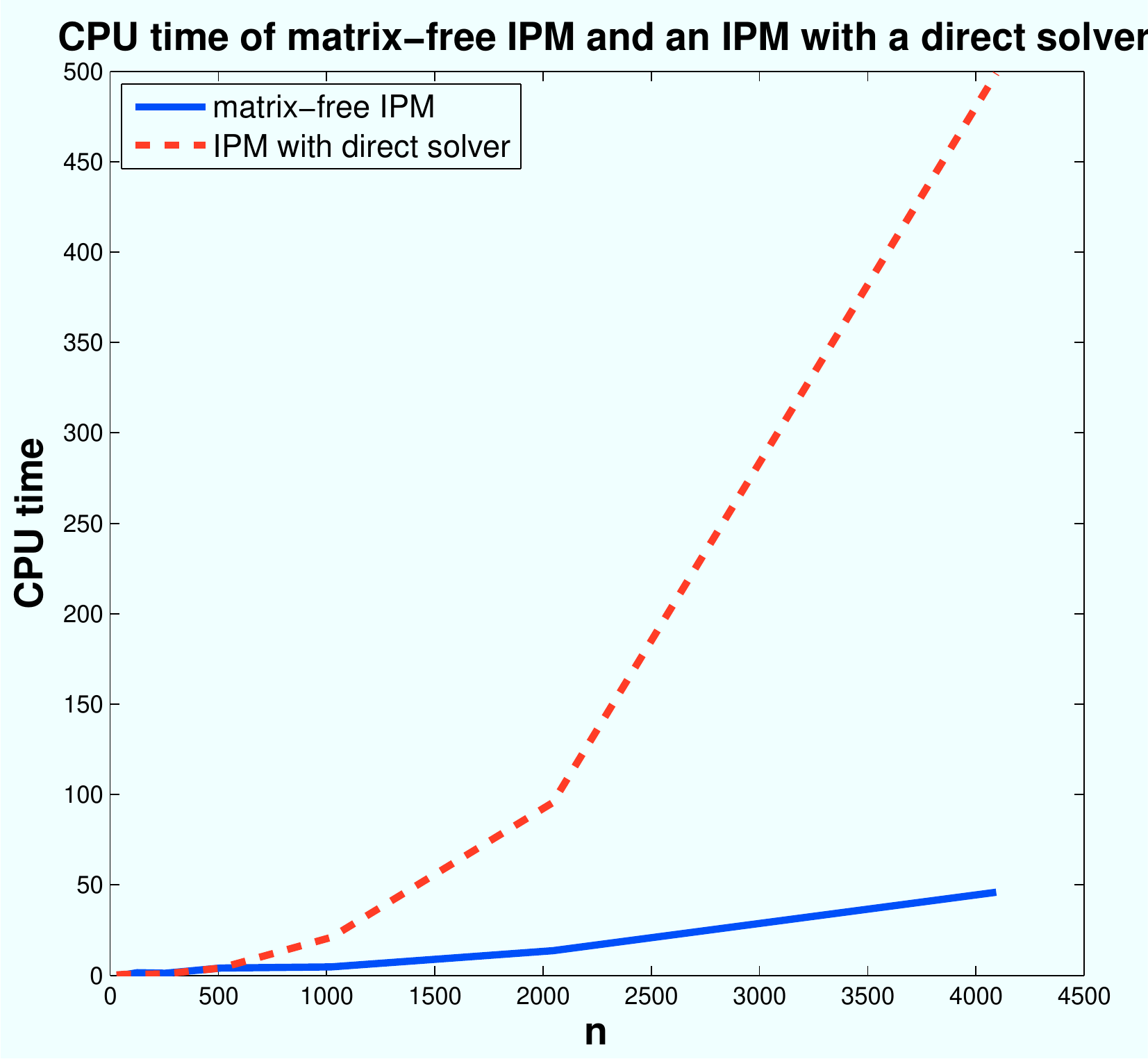}}
\subfloat[Scaling of number of iterations]{%
\label{fig2iters}%
\includegraphics[width=59mm,height=50mm]{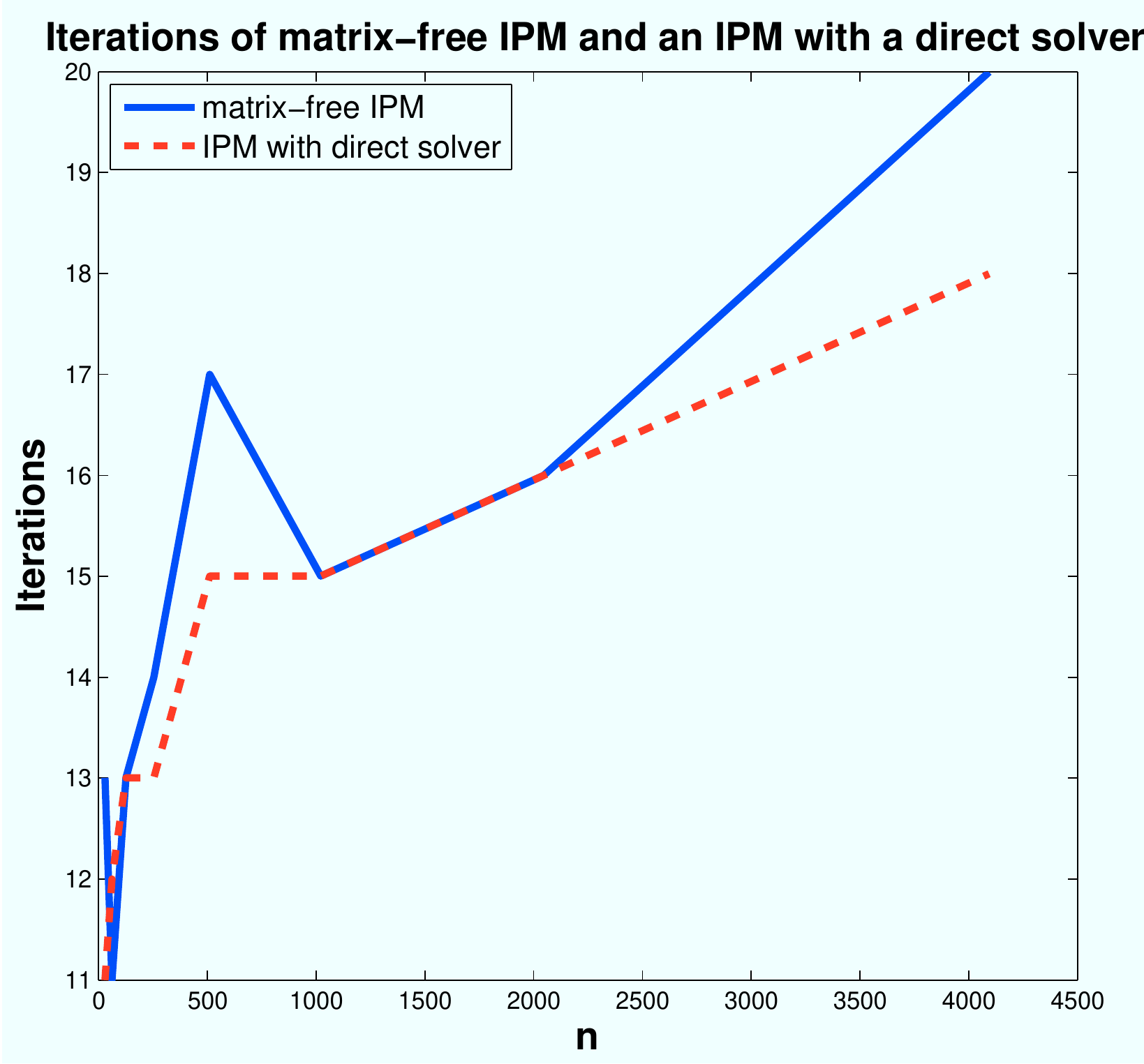}}\\
\caption{Scaling of CPU time and number of iterations as the size of problem $n$ grows for matrix-free IPM
and an IPM in which the PCG is replaced with a direct solver}
\label{figEigs}%
\end{figure}

\subsection{Average Phase Transition} \label{optransition}
\begin{figure}%
\centering
\subfloat{%
\includegraphics[scale=0.61]{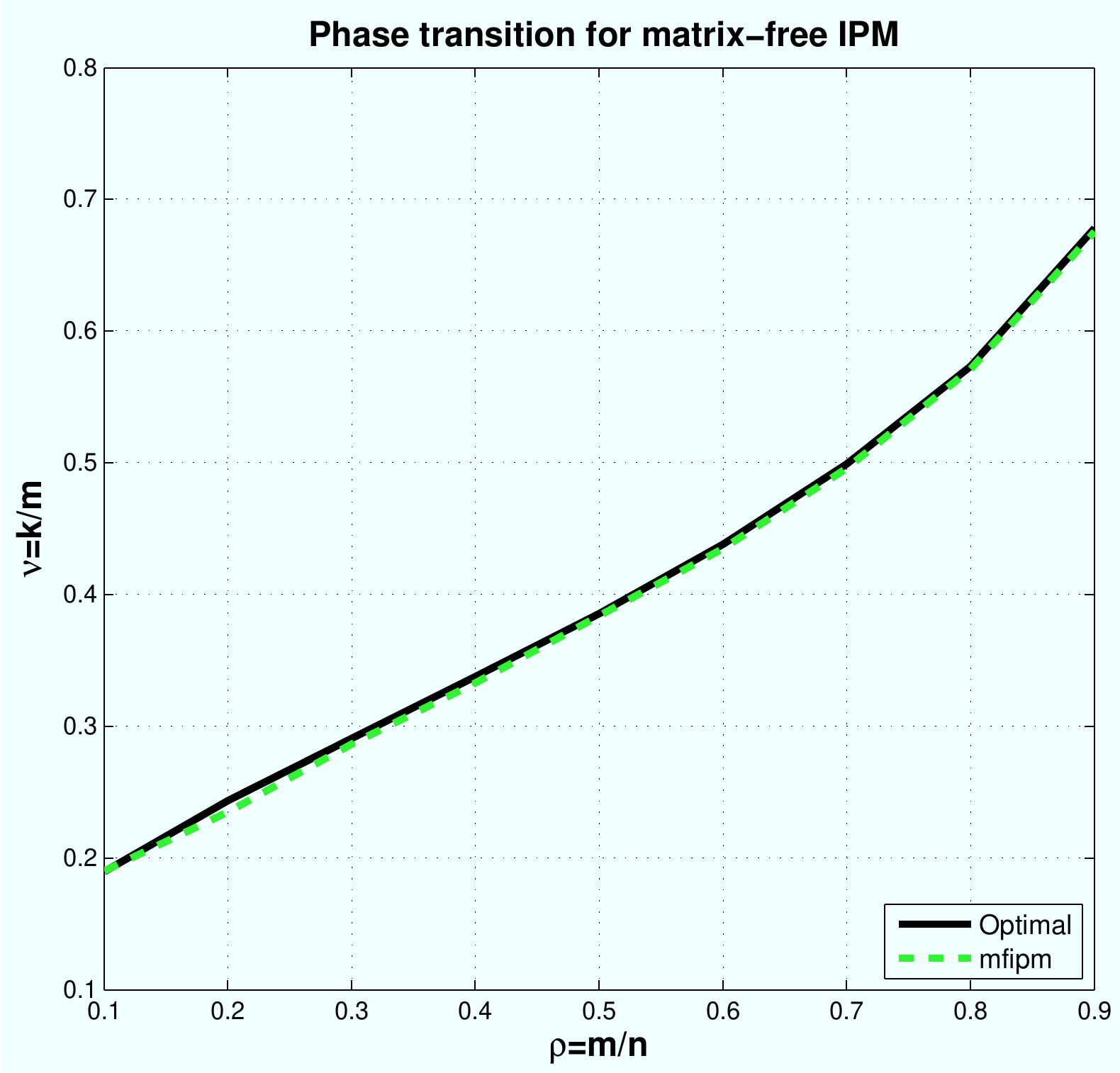}}
\caption{Empirical phase transition for matrix-free IPM. The solid curve denotes the theoretically optimal
         phase transition. The dashed curve denotes the empirical phase transition for $50\%$ success rate of matrix-free IPM}
\label{figPhase}%
\end{figure}
Recently, it has been shown in \cite{DonohoJarred:PUT} that for any problem instance ($A,b$), where $A$ is Gaussian, there is a maximum ratio $\bar{\nu}_{\rho}=k/m$ given
$\rho=m/n$ that below of it the problems (\ref{onenormform}) or (\ref{formulationsnoisy1}) guarantee on average
reconstruction of the optimal sparse representation.
The latter, has been introduced as the notion of average phase transition for Gaussian matrices. Moreover, it has been shown empirically that other measurement matrices
such as partial Fourier, partial Hadamard, Bernoulli etc, have the same average phase transition properties. Ideally, an efficient $\ell_1$-regularization solver should have empirical
average phase transition at the same level $\bar{\nu}_{\rho}$.

In this section we show that the empirical phase transition properties of matrix-free IPM fit the average Gaussian phase transition properties by reproducing a 
similar experiment as that in Section $2$ of \cite{DonohoJarred:PUT}. Let us now explain the experiment. 
The parameter $n$ is fixed to $n=1\thinspace000$. The measurements $m$ are varied from $m=100$ to $m=900$ with a step of $100$. 
For each of the nine measurements $m$ the sparsity of the optimal representation is varied from $k=1$ to $k=m$ with a
step of one and for each $k$, $100$ trials are conducted. The sensing matrix $A$ is chosen by taking randomly $m$ rows
from an $n\times n$ normalized discrete cosine transform matrix. Each nonzero coefficient of the sparse representation is set to $\pm1$
with equal probability, while the sparsity pattern is chosen at random. All the generated problems are solved using the matrix-free IPM solver, the reconstruction is considered successful when $\emph{r.e}\le$ 1.0e-5. 
For each ratio $\nu_{\rho}$ we compute the success ratio $p(\nu_{\rho})=S/100$, where $S$ is the number of trials for which
the $\emph{r.e}\le$ 1.0e-5. It has been demonstrated empiricaly in \cite{DonohoJarred:PUT} that for any problem instance ($A,b$), where $A$
is a partial DCT matrix 
a solver with average phase transition properties has $\displaystyle\max\{\nu_{\rho} \ | \ p(\nu_{\rho})\ge0.5 \}\approx \bar{\nu}_{\rho}$.
The latter means that the empirical average phase transition for $50\%$ success rate overlaps with the theoretical average phase transition for Gaussian matrices.
In Figure \ref{figPhase}, we plot the empirical phase transition for $50\%$ success rate of matrix-free IPM and the theoretical average phase transition. The two curves overlap.

\section{Conclusions}
\label{secCon}
We propose and implement a computationally inexpensive matrix-free primal-dual interior point method, 
based on \cite{IEEEhowto:Jacekmf} and \cite{mybib:G-MCorr}, for the \emph{$\ell_1$-regularized} problems
arising in the field of Compressed Sensing. At every iteration of the proposed primal-dual interior point method
the direction is obtained by solving the linear system (\ref{redNewtonA}) using the conjugate gradient method.
Unfortunately, the matrices $\Theta^{-1} + FF^\mathsf{T}$ in these systems tend to be ill-conditioned
as the algorithm converges, hence, the conjugate gradient method might get slow. To remedy this ill-conditioning we propose a low-cost preconditioner for the conjugate gradient method. 
The proposed preconditioning technique exploits features of Compressed Sensing matrices
as well as interior point methods. Its efficiency is justified theoretically and
confirmed in numerical experiments.

Computational experience presented in this paper shows that although the Compressed Sensing research community 
seems to favor first-order methods, a specialized (matrix-free) interior point method is very competitive and offers a viable alternative.


\bibliographystyle{plain}
\bibliography{csKFandJG.bib}

\end{document}